\documentclass{article}
\usepackage{amsmath, amsfonts, xypic, color, amsthm, amssymb, verbatim}
\usepackage[textsize=footnotesize]{todonotes}
\usepackage{ulem} 

\title{Computing Borel's Regulator}
\author{Zacky Choo, Wajid Mannan, Rub\'{e}n J.~S\'anchez-Garc\'ia \and Victor P.~Snaith}

\newcommand{\Tr}{\operatorname{Tr}}

\newtheorem{thm}{Theorem}[section]
\newtheorem{prop}[thm]{Proposition}
\newtheorem{cor}[thm]{Corollary}

\newtheorem{lem}[thm]{Lemma}
\theoremstyle{remark}
	\newtheorem{remark}{Remark}[section]
\theoremstyle{definition}
	\newtheorem{df}[thm]{Definition}
	\newtheorem{ex}{Example}

\newcommand{\norm}[1]{\ensuremath{\| #1 \|}}

\begin{document}

\maketitle


\abstract{\noindent  We present an infinite series formula based on the Karoubi-Hamida integral, for the universal Borel class evaluated on $H_{2n+1}(GL(\mathbb{C}))$.   For a cyclotomic field $F$ we define a canonical set of elements in $K_3(F)$  and present a novel approach (based on a free differential calculus) to constructing them. Indeed, we are able to explicitly construct their images in $H_{3}(GL(\mathbb{C}))$ under the Hurewicz map. Applying our formula to these images yields a value $V_1(F)$, which coincides with the Borel regulator $R_1(F)$ when our set is a basis of  $K_3(F)$ modulo torsion. For $F= \mathbb{Q}(e^{2\pi i/3})$ a computation of $V_1(F)$ has been made based on our techniques.}

{\let\thefootnote\relax\footnotetext{This research was funded by the EPSRC grant EP/C549074/1.}}

\bigskip
\section{Introduction}\label{section:Intro}

Let $F$ be an algebraic number field and $\mathcal{O}_F$ its ring of integers.
For $n\geq 1$, the Borel regulator $R_n(F)$ is a real valued numerical invariant of $F$.  It measures the covolume of the algebraic $K$-theory groups $K_{2n+1}(F)$ modulo torsion, embedded as a lattice in $\mathbb{R}^{d_n}$ (where the integer $d_n$ is to be specified later).  Although defined in terms of the odd-dimensional K-theory groups, knowledge of the Borel regulator has implications for the even dimensional K-theory groups of $F$ and $\mathcal{O}_F$.  In particular, the Lichtenbaum conjecture (proven in many cases, such as abelian extensions of the rationals \cite{BG03, HK03, KQDF, Sn04}) gives the order of $K_{2n}(\mathcal{O}_F)$ up to a power of 2 in terms of the Dedekind zeta function of $F$, the order of $\textup{Tor}\left(K_{2n+1}(F)\right)$ and $R_n(F)$.

However, explicitly computing the Borel regulator is a very difficult problem, even in the case of cyclotomic number fields.  The standard approach to the Borel regulator is via comparison with the Beilinson regulator \cite{BG02}, which in turn is expressed via polylogarithms; moreover, Zagier's conjectures \cite{Zagier91}, which generalize classical results of Bloch \cite{Bloch00, Grayson81}, allow one to map the higher Bloch group $\mathcal{B}_{n+1}(F)$ modulo torsion to a lattice in $\mathbb{R}^{d_n}$. However, the identification of $\mathcal{B}_{n+1}(F)$ modulo torsion with a full sublattice of $K_{2n+1}(F)$ is delicate \cite{BGdJarxiv11}.

We present a new approach to computing $R_1(F)$ when $F$ is a cyclotomic field.   We first describe a set of elements $z_u \in K_3(F)$ corresponding to primitive roots of unity $u \in F$  (Section \ref{section:BigCycle}).  We then explain how to compute the covolume $V_1(F)$ of the lattice generated by the images of these elements in the real vector space $\mathbb{R}^{d_1}$.  If the elements $z_u \in K_3(F)$ form a basis modulo torsion, then $R_1(F)=V_1(F)$ holds and we have computed the desired Borel regulator.  Otherwise $V_1(F)$ is an integer multiple of $R_1(F)$.

Determining for which cyclotomic fields the elements $z_u$ form a basis modulo torsion is beyond the scope of this article.  Nevertheless, our formula (Theorem \ref{thm:FinalSeries}) allows one in principle to compute $V_1(F)$ to any desired degree of accuracy.  In particular, one could numerically verify if $V_1(F) \neq 0$, in which case the elements $z_u$ would at least generate a full sub-lattice modulo torsion.


The elements $z_u \in K_3(F)$ have the advantage that we are able to express their images under the Hurewicz map explicitly in $H_3(GL(F))$.   Thus they are amenable to application of our formula for the universal Borel class (see below). Understanding how our elements relate to other elements in the $K$-theory of cyclotomic fields appears to be an intricate problem, and motivates further investigation.

Computing $V_1(F)$ breaks down into two stages.  Firstly (Section \ref{4}) we construct the images of the $z_u$ in $H_3(E(\mathbb{C}))$ as explicit chains ($E(\mathbb{C})$ denotes the discrete group of elementary matrices over the complex numbers). One of the difficulties in computing the Borel regulator of a field $F$ is that it is very hard to get such an explicit description of elements in the K--theory of $F$ or at least their images in $H_3(GL(F))$.  We overcome this (Section \ref{section:FreeFoxDifferentiation}) using novel techniques based on constructing a free differential calculus (motivated by the one invented by Fox as a tool in knot theory \cite{F60}).

Then  (Section \ref{SectionTwo}) we show how to apply the universal Borel class to these chains by expanding the Karoubi-Hamida integral \cite{Ha00} as an infinite power series (Theorem \ref{thm:FinalSeries}).  In fact, our formula works for any $n \ge 1$ and any number field $F$; it allows the computation of the Borel regulator class (defined in \S\ref{section:define}) applied to any element of $H_{2n+1}(GL(F))$ (Theorem \ref{thm:MainThm}).

At present the definition of $V_1(F)$ is limited to $F$ a cyclotomic field.  However, our introduction of a free differential and the methods based on it to describe elements in $K_3(F)$ are general (cf.~Remark \ref{rmk:generalFields}), and could in principle be applied to computing the Borel regulator of any field.  Further, the formula developed in Section \ref{SectionTwo} works for evaluating the universal Borel class on any number field (Theorem \ref{thm:MainThm}).  

Note that our elements lie in $K_3(F)$, so they do not allow the investigation of $R_n(F)$ for $n>1$.  In \S\ref{section:basis}, Remark \ref{Subsection:HighDim}, we indicate how one might generalize our construction to $n>1$, by considering the K--theory of spherical varieties. Observe that the formula for the Karoubi-Hamida integral works for any $n>1$ (Theorem \ref{thm:MainThm}). 

The theory developed in this paper has been implemented as a computer algorithm by the $1^{\rm st}$ author, in his PhD thesis \cite{ZackyPhD}.  This  gives in an actual estimate for $V_1(F)$ for the field  $F= \mathbb{Q}(e^{2\pi i/3})$, and it is described in Appendix \ref{section:Numerical}. We have also included some general computational aspects of our approach for anyone interested in an implementation (Appendix \ref{section:ComputerAlgorithm}).

\begin{remark}${}$
The $1^{\rm st}$ and $4^{\rm th}$ authors have results analogous to those in this article for the $p$-adic regulator \cite{CS10}.
\end{remark}

\section{Defining $V_1$ of a cyclotomic field} \label{section:BigCycle}

Let $F$ be the cyclotomic number field $\mathbb{Q}(\omega)$, where $\omega$ is a $q^{\rm th}$ root of unity ($q\geq3$).  A large part of the difficulty in computing the Borel Regulator of $F$ comes from the inaccessibility of elements of in the $K$-theory of $F$.  In this section we will define a canonical set of elements in $K_3(F)$, defined up to torsion (\S \ref{3.0}).   The images of this set under the Borel Regulator map generate a lattice and we define $V_1(F)$ to be its covolume (\S \ref{section:basis}).  

In Section \ref{4} we will then proceed to express our elements of $K_3(F)$ in the form necessary to compute $V_1(F)$.
We first recall the definition of the Borel regulator of a number field (\S\ref{section:define}).


\subsection{The Borel regulator}\label{section:define}
Let $F$ be a number field, $\mathcal{O}_F$ its ring of integers and $q_1, q_2$ the number of real embeddings, respectively conjugate pairs of complex embeddings, $F\hookrightarrow\mathbb{C}$.
The \emph{Borel regulator maps} are homomorphisms 
\begin{equation}\label{eq:BorelRegulatorMap}
  K_{2n+1}(\mathcal{O}_F) \cong K_{2n+1}(F)  \longrightarrow \mathbb{R}^{d_n} \qquad (n \ge 1)
\end{equation}
from the odd algebraic $K$-theory groups of $F$ (or $\mathcal{O}_F$) to $\mathbb{R}^{d_n}$, where $d_n$ is $q_1 + q_2$ if $n$ is even and $q_2$ if $n$ is odd. They can be defined in the following way (cf.~\S3.1.3 in \cite{Goncharov05}).

The Hurewicz homomorphism induces the following homomorphism from K-theory into the homology (with integral coefficients) of the discrete group $GL(\mathbb{C})$:
$$
\xymatrix{h_{2n+1}\colon K_{2n+1}(\mathbb{C}) = \pi_{2n+1}(BGL(\mathbb{C})^+) \ar[r] & H_{2n+1}(BGL(\mathbb{C})^+) \cong H_{2n+1}(GL(\mathbb{C}))\,.}
$$
\begin{remark}\label{remark:SuslinStability}
Suslin's stability result \cite[Corollary 2.5.3]{Kn01}, gives that if $N  \geq 2n+1$ then
\[
   H_{n}(GL(\mathbb{C})) \cong H_{n}(GL_N(\mathbb{C}))\,.
\]
\end{remark}
\noindent{}Let $\mathbb{R}(n) = (2\pi i)^n \mathbb{R}$ for $n \ge 1$. There exists a \emph{universal Borel class}  (see \cite{BG02} for a definition)
\[
	b_n \in H_c^{2n+1}\left( GL(\mathbb{C}); \mathbb{R}(n)\right)
\]
in the continuous cohomology of $GL(\mathbb{C}) = \text{colim}_r\, GL_r(\mathbb{C})$.   Application of $b_n$ induces a map $H_{2n+1}(GL(\mathbb{C})) \to \mathbb{R}(n)$.  The \emph{universal Borel regulator map}
\[
	r_n \colon K_{2n+1}(\mathbb{C}) \longrightarrow \mathbb{R}(n)
\]	
is defined to be the composition $r_n = b_ n\circ h_{2n+1}$.

For the definition on an arbitrary number field $F$, we compose with the maps induced on $K$-theory by the different embeddings $F \hookrightarrow \mathbb{C}$:
\[
	K_{2n+1}(F) \longrightarrow \bigoplus_{\text{Hom}(F,\mathbb{C})} K_{2n+1}(\mathbb{C}) \longrightarrow X_F \otimes \mathbb{R}(n),
\]
where $X_F = \mathbb{Z}[{\text{Hom}(F,\mathbb{C})}]$.  The image of this map is invariant under complex conjugation acting on both $\text{Hom}(F,\mathbb{C})$ and $\mathbb{R}(n)$.   Hence we have a map
\[
	K_{2n+1}(F) \longrightarrow \left(X_F \otimes \mathbb{R}(n)\right)^c
\]
where $(\ )^c$ denotes the subgroup of invariants under complex conjugation.   If $n$ is odd, we take a basis of $\left(X_F \otimes \mathbb{R}(n)\right)^c$ consisting of $\psi \otimes i -\overline{\psi}\otimes i$ for each conjugate pair of complex embeddings $\psi,\overline{\psi}\colon F \to \mathbb{C}$.   If $n$ is even we take a basis of $\left(X_F \otimes \mathbb{R}(n)\right)^c$ consisting of $\psi \otimes 1+\overline{\psi}\otimes 1$ for each conjugate pair of complex embeddings $\psi,\overline{\psi}\colon F \to \mathbb{C}$, together with $\psi \otimes 1$ for each real embedding $\psi\colon F\to \mathbb{R}$.  Either way, this yields a natural identification of $\left(X_F \otimes \mathbb{R}(n)\right)^c$ with $\mathbb{R}^{d_n}$. 

Borel proved that for $n \ge 1$ the Borel regulator map (\ref{eq:BorelRegulatorMap}) is an embedding of $K_{2n+1}(F)$ modulo torsion in $\mathbb{R}^{d_n}$ and its image is a full lattice in $\mathbb{R}^{d_n}$.  The covolume of this lattice is called the \emph{Borel regulator} for $F$, written $R_n(F)$.

\subsection{A canonical set of elements in $K_3(F)$} \label{3.0}

Let $R$ be a ring and let $E(R)$ denote the elementary matrices of $R$.  We will make use of the following result:

\begin{lem}\label{tinkerbell}  Let $R$ be a ring. Then the Hurewicz map induces a surjective map: $$
	\xymatrix{h_3\colon K_3(R) = \pi_3(BE(R)^+)\longrightarrow H_3(BE(R)^+)\cong H_3(BE(R)) \cong H_3(E(R))\,.}
$$
\end{lem}

\begin{proof}
By construction, the space $BE(R)^+$ is simply connected, so the surjectivity of $h_3$ follows immediately from the Hurewicz Theorem.
\end{proof}

Now let $F=\mathbb{Q}(\omega)$, where $\omega$ is a $q^{\rm th}$ root of unity ($q\geq3$).  
Let $E(F)$ denote the group of elementary matrices over $F$.   The kernel of $h_3\colon K_3(F)\to H_3(E(F)) $ is contained in the kernel of the Borel regulator map which contains only torsion elements.  Thus given  an element in $H_3(E(F))$, there exists a preimage in $K_3(F)$ (Lemma \ref{tinkerbell}), and it is unique up to torsion.  

We may therefore specify elements of $K_3(F)$ up to torsion, by giving cycles in the inhomogeneous bar resolution of $E(F)$ tensored with $\mathbb{Z}$, which we write as

$$ \cdots\stackrel d\longrightarrow B_n \stackrel d\longrightarrow B_{n-1} \stackrel d\longrightarrow \cdots
$$
For later reference we recall that the boundary map $d\colon B_3 \to B_2$ is given by $$ d( [g_{1} | g_{2} |  g_{3}])  =   [g_{2} |  g_{3}]    - [g_{1}  g_{2} |  g_{3}] +[g_{1} | g_{2}   g_{3}]-[g_1|g_2] $$
and the boundary map $d\colon B_2 \to B_1$ is given by $$ d( [g_{1} | g_{2}])  =   [g_{2}]  - [g_{1}g_{2}] +[g_{1}]. $$

Let $A=\mathbb{Z}[t, t^{-1}]$ be the ring of Laurent polynomials over $\mathbb{Z}$.  For any primitive $q^{\rm th}$ root of unity $u\in F$, we have a ring homomorphism $\alpha_u\colon A \to F$, given by $x \mapsto x|_{t=u}$, the element of $F$ obtained by evaluating $t$ at $u$.

For a unit $\lambda \in A$, let $g_{ij}^\lambda$ denote the matrix which differs from the identity in the $(i,i)^{\rm th}$ and $(j,j)^{\rm th}$ entries only, which are $\lambda$, $\lambda^{-1}$ respectively.  Let $a,b\in E(A)$ be 
\[ a = g_{12}^t=\left(
         \begin{array}{ccc}
           t & 0 & 0 \\
           0 & t^{-1} & 0 \\
           0 & 0 & 1 \\
         \end{array}
       \right),~
       b = g_{13}^{-t}=\left(
         \begin{array}{ccc}
           -t & 0 & 0 \\
           0 & 1 & 0 \\
           0 & 0 & -t^{-1} \\
         \end{array}
       \right).
\]
We write the inhomogeneous bar resolution of $E(A)$ as
$$ \cdots\stackrel {d'}\longrightarrow B_n' \stackrel {d'}\longrightarrow B_{n-1}' \stackrel {d'}\longrightarrow \cdots
$$
As $a,b$ commute, we have $d'([a|b]-[b|a])=0$.  Thus $[a|b]-[b|a]$ is a cycle and represents an element of $H_2(E(A))\cong K_2(A)$.  In fact (\cite[ pp. 71--75]{M71}) this element is trivial, so the cycle is actually the boundary of some $Z_1 \!\in\! B'_3$.  That is 
$$
	d'(Z_1)= [a|b]-[b|a].
$$

Let $h_3'\colon K_3(A) \to H_3(E(A))$ denote the surjective homomorphism of Lemma \ref{tinkerbell}.

\begin{lem} Given any other $Z_1'$ such that $d'(Z_1')= [a|b]-[b|a]$, the difference $Z_1'-Z_1$ is a cycle and represents a class $h_3'(y)\in H_3(E(A))$,  for some element $y \in K_3(A)$.  
\end{lem}

\begin{proof}
Clearly $d'(Z_1'-Z_1)=0$, so  $Z_1'-Z_1$ represents a homology class.  As $h_3'$ is surjective, there exists $y \in K_3(A)$ such that $h_3'(y)$ is this homology class.
\end{proof}

Thus $Z_1$ is a well defined chain modulo (chains representing) the image of $K_3(A)$.  An important point to note is that the construction of $Z_1$ made no reference to the cyclotomic field $F$ so we may make the following definition:

\begin{df}\label{def:Z1}
The \emph{universal chain} (defined up to chains representing elements in the image of $K_3(A)$) is the chain $Z_1 \in B_3'$ satisfying $d'(Z_1)= [a|b]-[b|a]$.
\end{df}
\begin{df}\label{Z2}
Let $q \ge 3$. We define $Z_2(q) \in B_3'$ as
\begin{equation*} 
Z_2(q) = \sum_{r=0}^{q-1} \left(   
  [a^{r}|a|b] 
-  [a^{r}|b|a]  
+   [b|a^{r}|a]   \right). 
\end{equation*}  
\end{df}
\noindent{}Then (writing $I$ for the identity matrix) we have
\begin{eqnarray*}
d'Z_2(q) &=&  \sum_{r=0}^{q-1} \left ( 
( [a|b] -[b|a]) -
 ([a^{r+1}|b]- [a^{r}|b])+(  [b|a^{r+1}]  - [b|a^{r}] )
\right) \\
&=&
q( [a|b] -[b|a]) -  ([a^{q}|b]-[I|b]) +
(  [b|a^{q}]  - [b|I] ).
\end{eqnarray*}
Note that $Z_2(q)$ is given explicitly above, whereas for $Z_1$ we have only proven existence and uniqueness up to the image of $K_3(A)$ under $h_3'$.

Let $\alpha_u$ denote the induced chain map:
\begin{eqnarray*}
\cdots\stackrel {d'}\longrightarrow &B_n'& \stackrel {d'}\longrightarrow B_{n-1}' \stackrel {d'}\longrightarrow \cdots\\
&\downarrow&\!\!\!\!\!\!\alpha_u   \qquad \downarrow \alpha_u \\
 \cdots\stackrel d\longrightarrow &B_n& \stackrel d\longrightarrow B_{n-1} \stackrel d\longrightarrow \cdots
\end{eqnarray*}
Note we use $\alpha_u$ to denote the induced maps on matrices, chains and elements of K-theory.
We have $$d \alpha_u(Z_2(q)) = \alpha_u d'Z_2(q) = q\alpha_u( [a|b] -[b|a])$$
as $\alpha_u(a)^q =I$.  

Let $Z=qZ_1-Z_2(q)$.  We then have $\alpha_u(Z) \in B_3$ is a cycle representing a homology class $\mathcal{Z}_u=[Z|_{t=u}]\in H_3(E(F))$. Its preimage under $h_3$, denoted $z_u\in K_3(F)$, is then determined up to torsion.

\begin{thm}
For each primitive $q^{\rm th}$ root of unity in the cyclotomic field $F=\mathbb{Q}(\omega)$, there exists  $z_u\in K_3(F)$, unique up to torsion, satisfying $h_3(z_u)=[(qZ_1-Z_2(q))|_{t=u}]$, where $Z_1$ is a representative of the universal chain (Definition \ref{def:Z1})  and $Z_2(q)$ is as given in Definition \ref{Z2}.
\end{thm}

\begin{proof}
Given a different choice of representative of the universal chain, $Z_1'$, we would have a homology class $\mathcal{Z}_u'=[\alpha_u(qZ_1'-Z_2(q))]$. Then $$\mathcal{Z}_u'-\mathcal{Z}_u =\alpha_u(h_3'(y))=h_3(\alpha_u(y))$$ for some $y \in K_3(A)$.  We may take $z_u'= \alpha_u(y)+z_u$ to be a preimage of $\mathcal{Z}_u'$ under $h_3$.

However, $$K_3(A) \cong K_3(\mathbb{Z}) \oplus K_2(\mathbb{Z})\cong \mathbb{Z}/2 \oplus \mathbb{Z}/48 \qquad{\rm (\cite[Theorem\, 5.3.30]{Rosenberg94})}.$$ Thus $K_3(A)$  is finite and  $\alpha_u(y) \in K_3(F)$ is an element of torsion as required.
\end{proof}

Thus modulo torsion we have a well-defined  canonical set of elements $z_u \in K_3(F)$, corresponding to the primitive $q^{\rm th}$ roots of unity $u \in F$.

\begin{remark}\label{rmk:generalFields}
Even if $F$ is a number field other than a cyclotomic field, we still have an element $z_u \in K_3(F)$ for each primitive root of unity $u \in F$.  However the number of pairs $u, u^{-1}$  of such roots of unity no longer coincides with the number of pairs of conjugate embeddings $ F\hookrightarrow \mathbb{C}$.  Thus we do not have a natural candidate for a basis for $K_3(F)$ (modulo torsion).
\end{remark}

\subsection{The covolume $V_1(F)$} \label{section:basis}

Let $l=\frac{\varphi(q)}2$ (where $\varphi$ is the Euler totient function)  and let $\pm v_1,\cdots, \pm v_l$ be the units of $\mathbb{Z}/q\mathbb{Z}$.  Also let $\xi= e^{2\pi i /q}$.   We have $l$ conjugate pairs of embeddings 
$\psi_j,\overline{\psi_j}\colon F = \mathbb{Q}(\omega) \hookrightarrow \mathbb{C}$, given by $\omega \mapsto \xi^{v_j}$ and $\omega \mapsto \xi^{-v_j}$ for $j=1,\cdots,l$.  We write $\psi_j$ to denote the induced map on chains and elements of K-theory.
The rank of $K_3(F)$ is $d_1= q_2=l$.

We also have $l$ conjugate pairs of primitive $q^{\rm th}$ roots of unity, $u_k, \overline{u_k}=\omega^{\pm v_k}\in F$ for $k=1,\cdots,l$.  Correspondingly we have elements $z_{u_k}, z_{\overline{u_k}}\in K_3(F)$ for $k=1, \cdots,l$. 

For $m$ a unit in $\mathbb{Z}/q\mathbb{Z}$, let $\beta_{m} \colon A \to \mathbb{C}$ be the ring homomorphism sending $t \mapsto \xi^{m}$.  Again we write $\beta_m$ to denote maps on chains and elements of K-theory.  Note that: $$\psi_j \circ \alpha_{u_k}= \beta_{v_jv_k}$$

Let $\left(L_{jk}\right)_{j,k}$ be the matrix of coordinates of the images of the $z_{u_k}$ under the Borel regulator map, $K_3(F) \to \mathbb{R}^{d_1} =  \left(X_F \otimes \mathbb{R}(n)\right)^c$, with respect to the basis $$\{\psi_j\otimes i - \overline{\psi}_j\otimes i\in\mathbb{R}^{d_1}\vert\,\, j=1,\cdots,l\}.$$

Let $Z=qZ_1-Z_2(q)$ as in \S\ref{3.0}.
\begin{lem}
We have $L_{jk}=b_1([Z|_{t=\xi^{v_jv_k}}])/i$.
\end{lem}

\begin{proof}
The coefficient $L_{jk}$ is  $r_1(\psi_j(z_{u_k}))/i$  (see \S \ref{section:define}).
We have
$$
\psi_j(\mathcal{Z}_{u_k})=[\psi_j\alpha_{u_k}(Z)]=[\beta_{v_jv_k}(Z)] \in H_3(E(\mathbb{C})).
$$
Thus
\begin{eqnarray*}
r_1(\psi_j(z_{u_k}))=b_1(h_3(\psi_j(z_{u_k})))=b_1(\psi_j(h_3(z_{u_k})))=b_1(\psi_j(\mathcal{Z}_{u_k}))\\=b_1[\beta_{v_jv_k}(Z)]=b_1([Z|_{t=\xi^{v_jv_k}}]).
\end{eqnarray*}
Hence for  $ 1\le j,k \le l$ we have $L_{jk}=r_1(\psi_j(z_{u_k}))/i=b_1([Z|_{t=\xi^{v_jv_k}}])/i$.  
\end{proof}

\begin{lem} \label{Humphrey}
Applying the Borel regulator map to $z_{u_k},\,\,z_{\overline{u_k}}\in K_3(F)$ gives elements of $\mathbb{R}^{d_1}$ which differ by a sign.  Thus up to torsion $z_{u_k},\,\,z_{\overline{u_k}}\in K_3(F)$ differ by a sign.  
\end{lem}

\begin{proof}
The invariance under complex conjugation of the image of $z_{u_k}$ in $ \left(X_F \otimes \mathbb{R}(n)\right)^c$, implies that  $r_1(\overline{\psi}_j(z_{u_k}))=-r_1(\psi_j(z_{u_k}))$ (see \S \ref{section:define}).   Thus 
\[
r_1(\psi_j(z_{\overline{u_k}}))= b_1[\beta_{-v_jv_k}(Z)]= r_1(\overline{\psi}_j(z_{u_k}))=-r_1(\psi_j(z_{u_k}))
.\qedhere
\]
\end{proof}

From Lemma \ref{Humphrey} we have that the following are well defined, independently of the choice of signs on the $v_k$:

\begin{df}
We define  $\textup{ind}(F)\in \mathbb{Z}$ to be the index of the lattice in $K_3(F)$ modulo torsion, generated by the $z_{u_k},\,\,\, k=1,\cdots,l$.  
\end{df}

\begin{df}
We define $V_1(F)$ to be the covolume of the lattice in $\mathbb{R}^{d_1}$ generated by the images of the  $z_{u_k},\,\,\, k=1,\cdots,l$, under the Borel Regulator map.
\end{df}
\begin{remark}
It is not a priori clear that the $z_{u_k}$ are non-zero elements of $K_3(F)$, let alone linearly independent.  However if one computes a non-zero value of $V_1(F)$ (see Appendix \ref{section:Numerical} for a reference to a computation which appears to to do this in the case $q=3$) it follows that 
the $z_{u_k}$ are linearly independent and generate a finite index subgroup of $K_3(F)$; in particular, $\textup{ind}(F) \neq 0$.
\end{remark}

\begin{thm} \label{almost there}
If $V_1(F) \neq 0$ then $\textup{ind}(F)\neq 0$ and the Borel regulator satisfies $R_1(F)=V_1(F)/\textup{ind}(F)$.
\end{thm}  

In particular, whenever $\textup{ind}(F)=1$ (that is, whenever the $z_{u_k}$ generate $K_3(F)$ modulo torsion), we are able to compute the Borel regulator $R_1(F)$.  Thus if criteria were found to determine for which cyclotomic fields $\textup{ind}(F)=1$, then this approach would allow one to compute the Borel regulator for those fields.

\begin{thm} \label{mothergoose} For a cyclotomic field $F= \mathbb{Q}(\omega)$, we may compute $V_1(F)$ in terms of the universal Borel class $b_1$:
\[
	V_1(F) = \left|\det \left(b_1 \left([(qZ_1-Z_2(q))|_{t=\xi^{v_jv_k}}]\right)\right)_{j,k}\right|
\]
where $Z_1$ is the universal chain (Definition \ref{def:Z1}) and $Z_2(q)$ as in Definition \ref{Z2}.
\end{thm}

\begin{proof}We know that  $V_1(F)$ is given by the absolute value of the determinant of the matrix $\left(L_{jk}\right)_{j,k}$ of coordinates of the images of the $z_{u_k}$. We need only note that factors of $i$ do not affect the absolute value of the determinant.
\end{proof}

Thus to compute $V_1(F)$ we must evaluate the $b_1([qZ_1-Z_2(q)]|_{t=\xi^{v_jv_k}})$.  This comprises two independent stages which are dealt with in Sections \ref{section:FreeFoxDifferentiation} and \ref{SectionTwo}.

Firstly, we need to explicitly construct the chain $qZ_1-Z_2(q)$.  As $Z_2(q)$ is given explicitly (Definition \ref{Z2}) it remains to find a representative of the universal chain $Z_1$.  Motivated by ideas in knot theory, in Section \ref{section:FreeFoxDifferentiation} we devise techniques for constructing elements in the bar resolution of a group, with specified boundary.  Using this, in Section \ref{4} we obtain an expression for $Z_1$. This expression consists of 6844 terms
 (chains of the form $(g_1|g_2|g_3)$ with each $g_i \in M_5(\mathbb Z[t,t^{-1}])$) and it is available as a PDF (for visualisation) or Matlab (for compuation) file\footnote{see ancillary files \textsf{anc/universalchain.tex} and \textsf{anc/universalchain.mat}}.  Although the construction of this expression is an intricate procedure (explained in \S4), one may easily calculate its boundary with the help of a computer and thus independently verify that $d'(Z_1)= [a|b]-[b|a]$, ie.~$Z_1$ is indeed a representative of the universal chain (cf.~Definition \ref{def:Z1}).

Secondly, in Section \ref{SectionTwo} we will show how the universal Borel class $b_n$ of \S\ref{section:define} can be computed for arbitrary $n\geq 1$ and any number field by making the Karoubi-Hamida integral \cite{Ha00} explicit.  In particular we describe a formula for computing $b_1$ (Theorem \ref{thm:FinalSeries}). 




\medskip
\begin{remark} \label{zedone}
An important point to note is that the universal chain $Z_1$ is defined independently of the integer $q$.  That is, having once computed it (as we do in Section \ref{4}), we may use it to compute $V_1(F)$ for any cyclotomic field $F$, merely by substituting in different roots of unity for the indeterminate $t$.

\end{remark}

\medskip
\begin{remark} \label{Subsection:HighDim}
To compute $R_n(F)$ for $n>1$ we would need to construct a basis modulo torsion of $K_{2n+1}(F)$. 
A potential idea for generalizing our methods in the $n=1$ case  is to let $S=F[x_{0}, \ldots , x_{r}]/(x_{0}x_{1} \ldots x_{r}( 1  -  \sum_{i=0}^{r} \ x_{i}))$, the coordinate ring of an algebraic sphere over $F$.  Then if we could construct elements in $K_3(S)$ we could  apply the natural homomorphism $K_3(S) \to KV_3(S)$ to obtain elements in Karoubi-Villamayor K-theory (\cite[\S3]{Weib82}).
There is an isomorphism \cite{DW80}:  $ KV_{3}(S)  \cong  K_{3}(F) \oplus  K_{3+r}(F)$. Projecting onto the second summand would yield elements of $K_{3+r}(F)$.
One may speculate that a procedure along these lines could be devised to find a basis (modulo torsion) for $K_{3+r}(F)$.  This would allow our  low dimensional group homology methods (\S\ref{2.6}) to be used for computing $R_n$, $n>1$.
\end{remark}

\section{Boundary relations in the bar resolution}\label{section:FreeFoxDifferentiation} 

As mentioned earlier, part of the intangibleness of the Borel Regulator of a field $F$ comes from the difficulty of explicitly constructing elements in the $K$-theory of $F$ in the required form: their images under the Hurewicz map represented by chains in the bar resolution of $GL(F)$.  In this section we describe a novel approach to doing precisely that, based on ideas from knot theory (specifically the concept of a free derivative).

In the Section \ref{4}, we will apply these general methods to the case of explicitly constructing the universal chain $Z_1$ (Definition \ref{def:Z1}), a necessary step in the calculation of $V_1(F)$ for a cyclotomic field, as explained before.

\subsection{A free Fox type derivative}\label{subsection:FreeFox}

Let $G$ be a discrete group and write $B_nG$ for the degree $n$ part of the inhomogeneous  bar resolution \cite[\S6.5]{W}. Therefore $B_nG$ is the free left ${\mathbb Z}[G]$-module with basis consisting of $n$-tuples $[g_{1} | g_{2} | \ldots | g_{n}]$ with each $g_{i} \in G$. 
The boundary map is given by
\begin{eqnarray*}
d( [g_{1} | g_{2} | \ldots | g_{n}]) &= &  g_{1}[g_{2} | \ldots | g_{n}] + \sum_{i=1}^{n-1}  \ (-1)^{i} [g_{1} | g_{2} | \ldots | g_{i}g_{i+1} | \ldots | g_{n}]   \\
& + & (-1)^{n} [g_{1} | g_{2} | \ldots | g_{n-1}] .  
\end{eqnarray*} 
Thus in particular the boundary map $B_3G \to B_2G$ is
\begin{eqnarray*}
d( [g_{1} | g_{2} |  g_{3}]) & = &  g_{1}[g_{2} |  g_{3}]    - [g_{1}  g_{2} |  g_{3}] +[g_{1} | g_{2}   g_{3}]-[g_1|g_2] .
\end{eqnarray*}

The study of knot theory motivated the idea of a free differential \cite{F60}:  a map from a set of words to an abelian object, satisfying a variant of Leibniz's rule.  We now construct the relevant such differential.  Let $F_G$ denote the free group on symbols $s_x$ with $x \in G$ and let $\phi \colon F_G \rightarrow G$ be the group homomorphism mapping $\phi(s_x) \mapsto  x$.  

\begin{df}\label{2.1}
There exists a free derivative:
\[   \partial\colon F_G \longrightarrow     {\mathbb Z}  \otimes_{{\mathbb Z}[G]}   B_{\,2}G    \] 
uniquely characterized by the properties:
\begin{itemize}
\item[(i)]  \  $\partial( e )=0$    \hspace{10pt}   ($e$ is the identity element of $F_G$) and

\item[(ii)]  \    $\partial( u s_x) = \partial(u) +   1   \otimes_{{\mathbb Z}[G]} [ \phi(u) | x ] $ for $u \in F_{G}$.
\end{itemize}
\end{df}

\begin{lem}{$_{}$}\label{2.2}
For $1 \leq i \leq r$ suppose $x_{i} \in G$ and that $\epsilon_{i} = \pm 1$. Set $z_{i} = x_{i}^{-1}$ if $\epsilon_{i}=-1$ and $z_{i} = 1$ otherwise. Then for any map $\partial$ satisfying (i) and (ii) as above, we have:
\begin{eqnarray}  
\partial(s_{x_{1}}^{\epsilon_{1}} s_{x_{2}}^{\epsilon_{2}} \ldots s_{x_{r}}^{\epsilon_{r}})
=   \sum_{i=1}^{r}  \  \epsilon_{i}    \otimes_{{\mathbb Z}[G]}   [ x_{1}^{\epsilon_{1}} x_{2}^{\epsilon_{2}} \ldots x_{i-1}^{\epsilon_{i-1}}z_{i} | x_{i} ].  \label{eqn:FoxSum}
\end{eqnarray}
Thus $\partial$ is uniquely defined on $F_{G}$, as claimed in the definition.
\end{lem}

\begin{proof}
 Firstly, note that (ii) implies 
\[ \partial( u) = \partial(u s_x^{-1} s_x ) = \partial(u s_x^{-1}) + 
1   \otimes_{{\mathbb Z}[G]}  [ \phi(u)x^{-1}  |  x ] \]
which means that
\[   \partial(u s_x^{-1}) = \partial( u) -  1   \otimes_{{\mathbb Z}[G]}  [ \phi(u)x^{-1}  |  x ]  .  \]
Induction on $r$ gives that 
$ \partial(s_{x_{1}}^{\epsilon_{1}} s_{x_{2}}^{\epsilon_{2}} \ldots s_{x_{r}}^{\epsilon_{r}})$ must be given by (\ref{eqn:FoxSum}).  It remains to verify that (\ref{eqn:FoxSum}) yields a well defined map $\partial\colon F_G \rightarrow     {\mathbb Z}  \otimes_{{\mathbb Z}[G]}   B_{\,2}G$, satisfying (i) and (ii).  

Two words represent the same element of $F_G$ precisely when they differ by a series of insertions and deletions of strings $s_ys_y^{-1}$ and $s_y^{-1}s_y$.  Direct calculation shows that (\ref{eqn:FoxSum}) is independent of such insertions or deletions.  Finally, direct calculation verifies that (\ref{eqn:FoxSum}) satisfies (i) and (ii).
\end{proof}

\begin{df}\label{2.3}
Given a word $w$ written in the form:
\[ w =  u_1 (s_{x_1} s_{y_1} s_{x_1 y_1}^{-1})^{n_1} u_1^{-1}   u_2 (s_{x_2} s_{y_2} s_{x_2 y_2}^{-1})^{n_2} u_2^{-1} \ldots  u_k (s_{x_k} s_{y_k} s_{x_k y_k}^{-1})^{n_k} u_k^{-1}   \]
with $x_i,y_i \in G$, $u_i \in F_G$ and $n_i \in \mathbb{Z}$, we define $W(w) \in    {\mathbb Z}   \otimes_{{\mathbb Z}[G]}  B_3(G) $ by the formula
\[ W(w) = \sum_{i=1}^k    \   n_i    \otimes_{{\mathbb Z}[G]}   [\phi( u_i)  | x_i  | y_i  ]  .  \]
Note $W$ is not a well defined function on $F_{G}$ (unlike the free derivative which is).
\end{df}

\begin{lem}{$_{}$} \label{2.4}  
The boundary of such a chain is given by the formula 
\[ (1 \otimes_{{\mathbb Z}[G]} d)(W(w)) =   \left({\sum_{i=1}^k   \   n_i \otimes_{_{{\mathbb Z}[G]}}  
[x_i | y_i]}\right) - \partial(w)   .  \]
\end{lem}

\begin{proof}
From (\ref{eqn:FoxSum}), we may make the following observation: \begin{eqnarray}  
\partial(uv)=\partial(u)+ \phi(u)\cdot \partial(v),\label{Lieb} \end{eqnarray}   
where for $g,h,k \in G$, it is understood that $g \cdot 1 \otimes_{_{{\mathbb Z}[G]}}  
[h | k]=1 \otimes_{_{{\mathbb Z}[G]}}  
[gh | k] $.

In the case of a segment  $u_i (s_{x_i} s_{y_i} s_{x_i y_i}^{-1})^{n_i} u_i^{-1}$, we have that:
\[   \phi(u_i (s_{x_i} s_{y_i} s_{x_i y_i}^{-1})^{n_i} u_i^{-1}) = \phi(u_i ) ({x_i} {y_i} {(x_i y_i)}^{-1})^{n_i} \phi(u_i )^{-1}   = e   \]

\noindent Regarding $w$ as the product of such segments and applying (3) we then have:
\[ \partial(w) = \sum_{i=1}^k  \   n_i   \partial(u_i (s_{x_i} s_{y_i} s_{x_i y_i}^{-1}) u_i^{-1}).   \] 
Let $u_i$ be written out as $\prod_{j=1}^{l_i}
s_{z_j}^{m_j}$, where $m_j$ is either 1 or $-1$.  Then by (\ref{eqn:FoxSum})
\begin{eqnarray*}
\lefteqn{\partial(u_i (s_{x_i} s_{y_i} s_{x_i y_i}^{-1}) u_i^{-1})}  \\
&=& \sum_{j  |   {m_j=1}}  \   1 \otimes_{{\mathbb Z}[G]}  [  \prod_{p=1}^{j-1}    z_p
^{m_p} | z_j ]   -    \sum_{j   |    {m_j=-1}}  \    1 \otimes_{{\mathbb Z}[G]}  [\prod_{p=1}^{j}    z_p ^{m_p} | z_j ]    \\
 &+& 1 \otimes_{{\mathbb Z}[G]} [\phi( u_i)  | x_i ] +  1 \otimes_{{\mathbb Z}[G]}   [ \phi( u_i) x_i  |  y_i]
  -  1 \otimes_{{\mathbb Z}[G]}    [\phi(u_i)  |  x_i y_i ]  \\
 &-& \sum_{j |  {m_j=1}} \    1 \otimes_{{\mathbb Z}[G]}  [  \prod_{p=1}^{j-1}    z_p
^{m_p}  |  z_j ]   +   \sum_{j  |  {m_j=-1}}   \   1 \otimes_{{\mathbb Z}[G]}  [ 
\prod_{p=1}^{j}    z_p ^{m_p}   |  z_j ]    \\
&=&1 \otimes_{{\mathbb Z}[G]}  [\phi( u_i)  | x_i ] +  1 \otimes_{{\mathbb Z}[G]}   ( \phi( u_i) x_i  |  y_i]
  -  1 \otimes_{{\mathbb Z}[G]}    [\phi(u_i)  |  x_i y_i ] .
\end{eqnarray*}
Thus we have
\begin{eqnarray*}
 \partial(w) 
  &=& \sum_{i=1}^k  \  n_i(1 \otimes_{{\mathbb Z}[G]}  [\phi( u_i)  | x_i ] +  1 \otimes_{{\mathbb Z}[G]}   [( \phi( u_i) x_i  |  y_i]
  -  1 \otimes_{{\mathbb Z}[G]}    [\phi(u_i)  |  x_i y_i ] ) 
\end{eqnarray*}
and
\vspace{-2ex}
\begin{eqnarray*}
(1 \otimes_{{\mathbb Z}[G]} d)(W(w)) &=& \sum_{i=1}^k  \  n_i(1 \otimes_{{\mathbb Z}[G]}  [x_i  | y_i ]- 1 \otimes_{{\mathbb Z}[G]}   [( \phi( u_i) x_i  |  y_i]\\
&+&      1 \otimes_{{\mathbb Z}[G]}    [\phi(u_i)  |  x_i y_i ] -1 \otimes_{{\mathbb Z}[G]}  [\phi( u_i)  | x_i ])\\ 
&=& \left(\sum_{i=1}^k  \  n_i(1 \otimes_{{\mathbb Z}[G]}  [x_i  | y_i ]\right) - \partial(w)
\end{eqnarray*}
as required. 
\end{proof}

\subsection{Constructing boundary relations}\label{2.6}

We now describe a method, based on Lemma \ref{2.4} for constructing boundary relations: identities of the form $(1 \otimes_{{\mathbb Z}[G]} d)\alpha = \beta$, for $\alpha \in  {\mathbb Z}   \otimes_{{\mathbb Z}[G]}  B_3(G)$ and $\beta \in  {\mathbb Z}   \otimes_{{\mathbb Z}[G]}  B_2(G)$.  

\begin{lem} \label{lemma:RelForm} Given $u \in F_G$ we may write 
\[u =  (s_{x_1} s_{y_1} s_{x_1 y_1}^{-1})^{n_1}  (s_{x_2} s_{y_2} s_{x_2 y_2}^{-1})^{n_2}  \ldots   (s_{x_k} s_{y_k} s_{x_k y_k}^{-1})^{n_k} s_{\phi(u)} \]
where $x_i,y_i \in G$.
\end{lem}

\begin{proof}
For contradiction let $l$ be the smallest integer such that there is some $u\in F_G$ of length $l$ contradicting the lemma.  As $e=(s_es_es_e^{-1})^{-1}s_e$, we know $l>0$.  Then either $u=u's_a$ or $u=u's_a^{-1}$ with $u'$ of length $l-1$ and $a \in G$.  Thus either 
$$u=Ps_{(\phi(u')}s_a= P(s_{(\phi(u')}s_a s_{\phi(u)}^{-1})s_{\phi(u)}$$
or 
$$u=Ps_{(\phi(u')}s_a^{-1}=P(s_{(\phi(u')}s_a^{-1}  s_{\phi(u)}^{-1})s_{\phi(u)}=P(s_{\phi(u)}s_a  s_{(\phi(u')}^{-1})^{-1}s_{\phi(u)},$$ 
where $P$ is a product of $(s_{x_i} s_{y_i} s_{x_i y_i}^{-1})^{n_i}$.
\end{proof}

\begin{df}\label{2.5}
A \emph{relator} is an element of ${\rm Ker}(\phi \colon  F_{G} \longrightarrow  G)$.
\end{df}

\begin{lem} \label{lem:rel}Given a relator $R \in F_G$ we may write 
\[R =  (s_{x_1} s_{y_1} s_{x_1 y_1}^{-1})^{n_1}  (s_{x_2} s_{y_2} s_{x_2 y_2}^{-1})^{n_2}  \ldots   (s_{x_k} s_{y_k} s_{x_k y_k}^{-1})^{n_k}\,.  \]
\end{lem}

\begin{proof}
We have 
\begin{eqnarray*}
\qquad \quad R &=&  (s_{x_1} s_{y_1} s_{x_1 y_1}^{-1})^{n_1}  (s_{x_2} s_{y_2} s_{x_2 y_2}^{-1})^{n_2}  \ldots   (s_{x_k} s_{y_k} s_{x_k y_k}^{-1})^{n_k}s_e \\ &=& (s_{x_1} s_{y_1} s_{x_1 y_1}^{-1})^{n_1}  (s_{x_2} s_{y_2} s_{x_2 y_2}^{-1})^{n_2}  \ldots   (s_{x_k} s_{y_k} s_{x_k y_k}^{-1})^{n_k}(s_es_es_e^{-1}).  \qquad \quad \qedhere 
\end{eqnarray*}
\end{proof}

Now let $C_1,\cdots,C_{k+1}$, $C_1',\cdots C_{l+1}'$  be products of commutators of the form $[R,u]^{\pm 1}=(RuR^{-1}u^{-1})^{\pm 1}$, where $R$ is a relator and $u\in F_G$.  Given an identity in $F_{G}$ of the form
\begin{multline}
\lefteqn{C_1v_1(s_{x_1}s_{y_1}s_{x_1y_1}^{-1})^{n_1}v_1^{-1} \cdots C_kv_k(s_{x_k}s_{y_k}s_{x_ky_k}^{-1})^{n_k}v_k^{-1} C_{k+1}}\\
 = C_1'v_1'(s_{x_1'}s_{y_1'}s_{x_1'y_1'}^{-1})^{n_1'}v_1'^{-1} \cdots C_l'v_l'(s_{x_l'}s_{y_l'}s_{x_l'y_l'}^{-1})^{n_l'}v_l'^{-1} C_{l+1}'
\label{eqn:FreeId}
\end{multline}
we may use Lemma \ref{lem:rel} to express each relator $R$ as a product of $(s_a s_b s_{ab}^{-1})^m$ and each $uR^{-1}u^{-1}$ as a product of $u(s_a s_b s_{ab}^{-1})^m u^{-1}$.  Thus we may express the left and right hand sides of 
(\ref{eqn:FreeId}) as words $w_1$, $w_2$ respectively, to which we may apply $W$.  We get
\[ 
	(1 \otimes_{{\mathbb Z}[G]} d)(W(w_1)) =   \left({\sum_{i=1}^k   \   n_i \otimes_{_{{\mathbb Z}[G]}}  
	[x_i | y_i]}\right) - \partial(w_1) 
\]
and
\[ 
	(1 \otimes_{{\mathbb Z}[G]} d)(W(w_2)) =   \left({\sum_{i=1}^l   \   n_i' \otimes_{_{{\mathbb Z}[G]}}  
	[x_i' | y_i']}\right) - \partial(w_2) 
\] 
as the remaining terms coming from each relator $R$ are canceled by the corresponding terms from each $uR^{-1}u^{-1}$.  

From (\ref{eqn:FreeId}) we have that $w_1=w_2$ as elements of $F_G$, so  $\partial(w_1) -\partial(w_2)=0$.
Thus:

\begin{thm} \label{BoundaryRelation}
We have a boundary relation:
\begin{eqnarray*} 
(1 \otimes_{{\mathbb Z}[G]} d)(W(w_1)-W(w_2)) =    \left({\sum_{i=1}^k   \   n_i \otimes_{_{{\mathbb Z}[G]}}  
[x_i | y_i]}\right) - \left({\sum_{i=1}^l   \   n_i' \otimes_{_{{\mathbb Z}[G]}} [x_i' | y_i']}\right) .
\end{eqnarray*}
\end{thm}

\subsection{Examples}\label{subsection:examples}

Let $G$ be a group and let $x,y \in G$ commute.  Then $w=[s_x,s_y]$ is a relator and we have $w=(s_xs_ys_{xy}^{-1})(s_ys_xs_{yx}^{-1})^{-1}$.  Then $W(w)= 1    \otimes_{{\mathbb Z}[G]}   [e  | x  | y  ] - 1    \otimes_{{\mathbb Z}[G]}   [e  | y  | x  ]$ and \[(1 \otimes_{{\mathbb Z}[G]}  d)W(w)= 1    \otimes_{{\mathbb Z}[G]}   [ x  | y  ] - 1    \otimes_{{\mathbb Z}[G]}   [ y  | x  ] - \partial(w).\]  We use the notation  $\{ x ,  y \}$ to denote $1  \otimes_{{\mathbb Z}[G]}   [x|y]   -  1  \otimes_{{\mathbb Z}[G]}   [y|x]$. 

\begin{ex}\label{eg:IdCom}
As our first example we consider the identity \[(s_es_ys_y^{-1})(s_ys_es_y^{-1})^{-1} =[s_e,s_y].\]  Letting $w_1,w_2$ denote the left and right sides of this identity as before, we get
\[
	W(w_1)-W(w_2)= 1    \otimes_{{\mathbb Z}[G]}   [e  | e  | y  ] - 1    \otimes_{{\mathbb Z}[G]}   [e  | y  | e  ] - 1    \otimes_{{\mathbb Z}[G]}   [e  | e  | e  ] + 1    \otimes_{{\mathbb Z}[G]}   [y  | e  | e  ]\,. 
\]   
Thus by Theorem \ref{BoundaryRelation}
\[
	(1 \otimes_{{\mathbb Z}[G]}  d)(W(w_1)-W(w_2))= 1    \otimes_{{\mathbb Z}[G]}   [ e  | y  ]
-1    \otimes_{{\mathbb Z}[G]}   [ y  | e  ]= \{ e ,  y \}\,.
\]
In fact  $(1 \otimes_{{\mathbb Z}[G]}  d)\left(1    \otimes_{{\mathbb Z}[G]}   [e  | e | y  | e  ] + 1    \otimes_{{\mathbb Z}[G]}   [e |e  | e  | e  ]\right)
=1    \otimes_{{\mathbb Z}[G]}   [e  | y  | e  ] + 1    \otimes_{{\mathbb Z}[G]}   [e  | e  | e  ]$.   So adding this boundary of a 4-chain to $W(w_1)-W(w_2)$ leaves the boundary relation
\[  (1 \otimes_{{\mathbb Z}[G]}  d)\left(1    \otimes_{{\mathbb Z}[G]}   [ e | e  | y  ]
+1    \otimes_{{\mathbb Z}[G]}   [ y  | e | e  ]\right) = \{ e ,  y \}.\]
\end{ex}

\begin{ex}\label{2.7}
Let $x, y,c \in G$ satisfy $[x,c]=[y,c]=e$.  Consider the identity
\[   s_{x}  [  s_{y} , s_{c}] s_x^{-1}  [ s_{x} , s_{c}] = [  s_{x} s_{y}, s_c]=(s_xs_ys_{xy}^{-1})[s_{xy},s_c]s_c(s_xs_ys_{xy}^{-1})^{-1}s_c^{-1}.  \]
Let $w_1,w_2$ denote the left and right hand sides of this identity. Then
\begin{eqnarray*}
 W(w_1)    &=& 1  \otimes_{{\mathbb Z}[G]}  [x|y|c]  -   1  \otimes_{{\mathbb Z}[G]}  [x|c|y)]  
      +  1   \otimes_{{\mathbb Z}[G]}   [e|x|c]   -   1   \otimes_{{\mathbb Z}[G]}  [e|c|x]  
\end{eqnarray*}
and
\begin{eqnarray*}
W(w_2)    &=&      1  \otimes_{{\mathbb Z}[G]}  [e|xy|c]  -  1   \otimes_{{\mathbb Z}[G]}  [e|c|xy] 
       +  1  \otimes_{{\mathbb Z}[G]}   [e|x|y]    -  1  \otimes_{{\mathbb Z}[G]}  [c|x|y] .
\end{eqnarray*}
Thus by Theorem \ref{BoundaryRelation} we have the boundary relation
\[  ( 1  \otimes_{{\mathbb Z}[G]}  d)(W(w_1)-W(w_2)) =   \{ x ,  c \} +  \{ y ,  c \} -  \{ xy ,  c \}  . \]
Adding the boundary of the 4-chain \[
 1   \otimes_{{\mathbb Z}[G]}  [e | x | y | c] +1   \otimes_{{\mathbb Z}[G]}  [e | c |x | y] - 1  \otimes_{{\mathbb Z}[G]}  [e | x | c | y]
\]
to $W(w_1)-W(w_2)$ we get a 3-chain with the same boundary:
\begin{align*}
 ( 1  \otimes_{{\mathbb Z}[G]}  d)\left( 1  \otimes_{{\mathbb Z}[G]}  [x|y|c]  -   1  \otimes_{{\mathbb Z}[G]}  [x|c|y)]    +  1  \otimes_{{\mathbb Z}[G]}   [c|x|y] \right)   \\
   =   \{ x ,  c \} +  \{ y ,  c \} -  \{ xy ,  c \}  .  
\end{align*}
\end{ex}

\medskip

\begin{ex}\label{2.9}
Let $x, y, c \in G$ satisfy $cxc^{-1} = y, cyc^{-1} = x, xy = yx$.  Then we have
\begin{eqnarray*}
\lefteqn{s_{c}(s_{x} s_{y} s_{xy}^{-1}) (  s_{y} s_{x} s_{yx}^{-1})^{-1}  s_c^{-1}}\\
&=& [ (s_{c}s_{x}s_{cx}^{-1})(s_ys_cs_{yc}^{-1})^{-1} ,   s_{y} s_{c} s_{y}s_{c}^{-1} s_{y}^{-1}  ]
( s_{y}s_{x}s_{yx}^{-1}) ( s_{x} s_{y} s_{xy}^{-1})^{-1} \\ 
&& [ s_{x} s_{y}s_{x}^{-1} , ( s_{c} s_{y}s_{cy}^{-1})(s_xs_cs_{xc}^{-1})^{-1}  ]  .
\end{eqnarray*}
Again let $w_1,w_2$ denote the left and right hand sides of this identity.  Theorem \ref{BoundaryRelation} gives \[    ( 1  \otimes_{{\mathbb Z}[G]}  d)(W(w_1) - W(w_2)) =  2  \{ x ,  y \}\,.   \]
Applying $W$ to $w_1,w_2$ gives
\begin{eqnarray*}
 W(w_1)   &=&  1  \otimes_{{\mathbb Z}[G]}  [c|x|y]  -   1  \otimes_{{\mathbb Z}[G]}  [c|y|x)]   \\
W(w_2)    &=&       1  \otimes_{{\mathbb Z}[G]}  [e|c|x]  -  1   \otimes_{{\mathbb Z}[G]}  [e|y|c]  \\
& & -\  1  \otimes_{{\mathbb Z}[G]}   [x|c|x]    +  1  \otimes_{{\mathbb Z}[G]}  [x|y|c] \\
&&+\   1  \otimes_{{\mathbb Z}[G]}   [e|y|x]    -  1  \otimes_{{\mathbb Z}[G]}  [e|x|y] \\
&&+\  1  \otimes_{{\mathbb Z}[G]}   [y|c|y]    -  1  \otimes_{{\mathbb Z}[G]}  [y|x||c] \\
&& - \  1  \otimes_{{\mathbb Z}[G]}   [e|c|y]    +  1  \otimes_{{\mathbb Z}[G]}  [e|x|c] \,.
\end{eqnarray*}
We have a 4-chain 
\begin{eqnarray*}
   T &=& 1  \otimes_{{\mathbb Z}[G]}  [e|y|x|c] +1  \otimes_{{\mathbb Z}[G]}[e|x|c|x] +1  \otimes_{{\mathbb Z}[G]}[e|c|x|y] \\
&& -\ 1  \otimes_{{\mathbb Z}[G]}  [e|x|y|c] -1  \otimes_{{\mathbb Z}[G]}[e|y|c|y] -1  \otimes_{{\mathbb Z}[G]}[e|c|y|x]\,.
\end{eqnarray*}
Let
\begin{eqnarray*}
P &=&  W(w_1) - W(w_2) +( 1  \otimes_{{\mathbb Z}[G]}  d)T\\
&=& 1\!  \otimes_{{\mathbb Z}[G]} \! [c | x|y] - 1\!  \otimes_{{\mathbb Z}[G]}\!  [x | y | c]  -  1 \! \otimes_{{\mathbb Z}[G]}\!  [c | y | x] \\ 
&& +\  1\!  \otimes_{{\mathbb Z}[G]}\!  [y | x | c]  + 1\!  \otimes_{{\mathbb Z}[G]}\! [x | c | x ]  
  - 1 \! \otimes_{{\mathbb Z}[G]}\![y | c | y] .   
\end{eqnarray*}
\noindent Thus we have a boundary relation $( 1  \otimes_{{\mathbb Z}[G]}  d)(P)= 2\{x,y\}$.
\end{ex}

\section{Constructing the universal chain}\label{4}

As in \S\ref{3.0} let $A=\mathbb{Z}[t,t^{-1}]$ and $a=g_{12}^t$, $b=g_{13}^{-t}$.  Our aim in \S\ref{3.1} is to describe the idea behind the construction of the universal chain $Z_1$, which satisfies $d'(Z_1)=[a|b]-[b|a]$.  Then in \S\ref{3.2} we describe the actual process of constructing it.  Recall (Remark \ref{zedone} on page \pageref{zedone}) that once constructed, we may use it to compute $V_1(F)$ for any cyclotomic field $F$, by substituting in the relevant roots of unity for the indeterminate $t$.

\subsection{The strategy for the universal chain}\label{3.1}

In order to construct this boundary relation we will employ the method given in \S\ref{2.6} and used in Examples \ref{eg:IdCom}, \ref{2.7}, \ref{2.9} of \S\ref{subsection:examples}.  In this case our group is $E(A)$ and we seek an identity in the letters of the free group generated by elements of $E(A)$: 
\begin{align} C_1(s_as_b s_{ab}^{-1})(s_bs_a s_{ab}^{-1})^{-1}C_2=C_3 \label{eqn:BigCycle}\end{align}
where, as before $C_1,C_2,C_3$ are products of commutators $[R,u]^{\pm 1}$, with $R$ a relator.

In this subsection we will describe the idea behind the construction of (\ref{eqn:BigCycle}).  Then in \S\ref{3.2} we will go though the stages in its construction.

For $i \neq j$ and $\mu \in A$, let $E_{ij}^\mu \in E(A)$ differ from the identity in $E(A)$ in the $(i,j)^{\rm th}$ entry only, which is $\mu$.  Let $F_E\!\subset\! F_{E(A)}$ be the subgroup generated by the $s_{E_{ij}^\mu }$.

The Steinberg group $St(A)$ is defined to be the group generated by letters $X_{ij}^{\mu}$, subject to the Steinberg relations $S_{ij}^{\mu,\nu}= T_{ijk}^{\mu,\nu}=U_{ijkl}^{\mu,\nu}=e$, where
\[
\begin{array}{rcll}
S_{ij}^{\mu,\nu} &=& (X_{ij}^{\mu+\nu})^{-1} X_{ij}^\mu X_{ij}^\nu,   & i \neq j \\
T_{ijk}^{\mu,\nu} &=&X_{ij}^\mu X_{jk}^\nu (X_{ij}^{\mu})^{-1} (X_{jk}^{\nu})^{-1} (X_{ik}^{\mu\nu})^{-1},&  i,j,k \, {\rm distinct}\\
U_{ijkl}^{\mu,v} &=&X_{ij}^\mu X_{kl}^\nu (X_{ij}^{\mu})^{-1}
(X_{kl}^{\nu})^{-1},&  i \neq l, j \neq k, i \neq j, k \neq l .
\end{array}
\]

The homomorphism $\psi\colon St(A) \to E(A)$ mapping $X_{ij}^{\mu} \mapsto E_{ij}^{\mu}$ is clearly surjective and its kernel $K_2(A)$ is central in $St(A)$ (\cite[p.40, Theorem 5.1]{M71}).  Thus given $x,y \in E(A)$ and preimages $x',y' \in St(A)$, the commutator $[x',y']$ is independent of the choice of preimages, and may be denoted $\{x,y\}$ (this differs from our earlier notation).

The first step in our construction is to express $\{a,b\}$ explicitly as a product of the $X_{ij}^{\mu}$.  This is done by writing $a$ and $b$ as products of the $E_{ij}^{\mu}$, then replacing each $E_{ij}^{\mu}$ with $X_{ij}^{\mu}$, to get $a', b' \in St(A)$.  Then $\{a,b\}$ is represented by the word $[a',b']$.

In  \cite[ pp. 71--75]{M71} it is shown that $\{a,b\}=e$ in $St(A)$.   This means that we have an identity in the free group generated by the letters $X_{ij}^\mu$:
\begin{eqnarray}
[a',b']= (w_1 R_1 ^{\pm 1} w_1^{-1} ) (w_2 R_2 ^{\pm 1} w_2^{-1} ) \cdots (w_m R_m ^{\pm 1} w_m^{-1} ) \label{eqn:SteinId}
\end{eqnarray}
where the $R_i$ are of the form $S_{ij}^{\mu,\nu}$, $T_{ijk}^{\mu,\nu}$ or $U_{ijkl}^{\mu,\nu}$.

Let $R_E \subset F_E$ be the subgroup ker$(\phi\vert_{F_E})$, and let $\hat{\theta}\colon F_E \to St(A)$ be given by $s_{E_{ij}^\mu} \mapsto X_{ij}^\mu$.  Then $\hat{\theta}$ is surjective and by construction the following diagram commutes:\[
\xymatrix{
F_E\ar[dr]_-{\phi\vert_{F_E}}\ar[r]^-{\hat{\theta}}&St(A)\ar[d]^-\psi\\
&E(A)\,.
}\]
Hence $\hat\theta(R_E) \subset {\rm ker}(\psi)$ which is, as we have said, central.  Hence $\hat\theta([F_E,R_E])=\{e\}$ so $\hat \theta$ induces a well defined map $\theta\colon F_E/[F_E,R_E] \to St(A)$ mapping $s_{E_{ij}^\mu} \mapsto X_{ij}^\mu$.  

Note that $[F_E,R_E]\subset R_E$ so $\phi_{F_E}$ induces a well defined map $\phi \colon F_E/[F_E,R_E]\to E(A)$ and the following diagram also commutes:\[
\xymatrix{
F_E/[F_E, R_E]\ar[dr]_-{\phi}\ar[r]^-{{\theta}}&St(A)\ar[d]^-\psi\\
&E(A)
}\]
where $\phi$ here is understood to denote the map induced by the restriction.

The kernel of $\theta$ is contained in ker$({\psi\theta})=R_E/[F_E,R_E] $, so is central in $F_E/[F_E,R_E]$.  Thus $\theta$ is a central extension.  From \cite[pp.48--51]{M71}  applied to $\theta$ we get:

\begin{lem}  

i) The element $[s_{E_{ik}^1}, s_{E_{kj}^\mu}]\in F_E/[F_E,R_E]$ is independent of $k\neq i,j$.

 ii) The map $\theta$ has a section $c$, given by $c(X_{ij}^\mu)= [s_{E_{ik}^1}, s_{E_{kj}^\mu}]$, $k\neq i,j$.
\end{lem}

In particular $c$ respects the Steinberg relations.  Thus each $c(R_i)=e\in F_E/[F_E,R_E]$ and may be expressed as a product of commutators $[u,v]^{\pm 1}$, with $v\in R_E$.   

We apply $c$ to (\ref{eqn:SteinId}) and take conjugation by $c(w_i)$ inside the commutators to get
\[[c(a'),c(b')]= C_3\] 
where $C_3$ has the required form of a product of commutators, each involving a relator.  We have $\phi c (a')=\psi\theta c(a')=\psi(a')=a$ and $\phi c (b')=\psi\theta c (b')=\psi(b')=b$, so $K_a=c(a')s_a^{-1}$ and $K_b=c(b')s_b^{-1}$ are relators.  We have
\[ [c(a'), c(b') ] =  [K_a s_a, K_b s_b]=[K_a, s_a s_bs_a^{-1}](s_a s_b
s_{ab}^{-1})(s_b s_a s_{ab}^{-1})^{-1}[s_b K_a s_a s_b^{-1},K_b]. \] 
Let $C_1=[K_a, s_a s_bs_a^{-1}]$ and $C_2=[s_b K_a s_a s_b^{-1},K_b]$.   Then we have constructed (\ref{eqn:BigCycle}) as
\begin{eqnarray} C_1(s_a s_b
s_{ab}^{-1})(s_b s_a s_{ab}^{-1})^{-1}C_2=[c(a'), c(b') ]=C_3. \label{eqn:LongWord}\end{eqnarray}

\subsection{Computing the free group identity }\label{3.2}
We shall now describe how to implement the strategy of \S\ref{3.1}.  The first step is to find $a',b' \in St(A)$.  For $\lambda$ a unit in $A$, let $Y_{ij}^\lambda =
X_{ij}^\lambda X_{ji}^{-(\lambda^{-1})} X_{ij}^\lambda$.  Factorizing $g_{ij}^\lambda$ into matrices of the form $E_{ij}^\mu$ and replacing each $E_{ij}^\mu$ with $X_{ij}^\mu$, we get $Y_{ij}^\lambda Y_{ij}^{-1}$, which we will denote $h_{ij}^\lambda$.  Thus $\psi(h_{ij}^\lambda)=g_{ij}^\lambda$ and in particular $a'=h_{12}^t, b'=h_{13}^{-t}$.

In \cite[pp.~71--75]{M71}, it is shown that the commutator $[h_{12}^t,h_{13}^{-t}]$ may be reduced via the Steinberg relations to $e$.  This proof depends on the identities $h_{13}^{-t}Y_{12}^\lambda (h_{13}^{-t})^{-1} = Y_{12}^{-t \lambda}$
for $\lambda = t$ or $-1$, and $Y_{12}^t Y_{12}^{-1} Y_{12}^{-t}
= Y_{12}^{-t^2}$, which themselves depend on several identities proved in \cite[Lemma 9.2]{M71}.  The proof of each of these is given by taking words in the $X_{ij}^\mu$ and simplifying them using the Steinberg identities. 

Thus the entire proof may be written out as a sequence of equalities in the Steinberg group, where at each step a simplification is made using one of the Steinberg identities.  Of course such a proof would be extremely long, as each step in proving an identity needs to be repeated every time the identity is used to prove a consequential identity.  If the consequential identity is used several times  to prove an identity higher up the chain of consequences then one can appreciate how the length of such a proof grows exponentially with the length of the proof given in \cite{M71}.

Thus we have a long chain of equalities in the Steinberg group $[h_{12}^t,h_{13}^{-t}]=\cdots=e$, where at each step we essentially factor off a conjugate of one of the relators $S_{ij}^{\mu,\nu}$, $T_{ijk}^{\mu,\nu}$,  $U_{ijkl}^{\mu,\nu}$ or their inverses.   Next we must write the entire proof out again, this time not suppressing these factors, so we have a sequence of equalities in the free group on the letters $X_{ij}^\mu$.  This long nested sequence of operations, together with the vast amount of data needed to store all these factors, made it natural to employ a computer to construct the resulting identity:
\begin{eqnarray} [a',b']=[ h_{12}^{t}, h_{13}^{-t} ] = \cdots= \prod_{i=1}^m (w_i R_i ^{\pm 1}  w_i^{-1} ) \label{eqn:ProofCom} \end{eqnarray}
where the $R_i$ are words of the form $S_{ij}^{\mu,\nu}$, $T_{ijk}^{\mu,\nu}$ or $U_{ijkl}^{\mu,\nu}$ and $m = 392$. The sequence of letters in the product would fill about 40 pages. We note that none of the
relations used required us to include extra indices, so only 1, 2, and 3
were used.

We next apply the homomorphism $c$.  Explicitly, this means replacing each $X_{ij}^{\mu}$ in (\ref{eqn:ProofCom}) by $[s_{E_{i4}^{1}}, s_{E_{4j}^\mu} ]$.  This is now a free group identity in $F_E\subset F_{E(A)}$:
\begin{eqnarray} [c(a'),c(b')]= \prod_{i=1}^m (c(w_i) c(R_i) ^{\pm 1}  c(w_i)^{-1} )\,. \label{eqn:FreeIdCom} \end{eqnarray}

From \cite[pp.~48-51]{M71} we know that $c\colon St(A) \to F_E/[F_E,R_E]$ is a well defined homomorphism.  
Thus as the words $S_{ij}^{\mu,\nu}$, $T_{ijk}^{\mu,\nu}$ and $U_{ijkl}^{\mu,\nu}$ represent $e \in  St(A)$, we have that the words $c(S_{ij}^{\mu,\nu}), c(T_{ijk}^{\mu,\nu}),c(U_{ijkl}^{\mu,\nu}) \in F_E$ represent  $e \in F_E/[F_E,R_E]$.

The proofs of these three identities \cite[pp.49-51]{M71} can be written as a sequence of equalities in $F_E/[F_E,R_E]$, where at each step we factor off a word in $[F_E,R_E]$.  As before, by not suppressing these factors we obtain identities in the free group $F_E$, equating the $c(S_{ij}^{\mu,\nu}), c(T_{ijk}^{\mu,\nu}),c(U_{ijkl}^{\mu,\nu}) \in F_E$ with elements of $[F_E,R_E]$.  For example:
\begin{eqnarray*}
c(U_{ijkl}^{\mu,\nu})=\left[ [s_{E_{i4}^{1}}, s_{E_{4j}^\mu} ], [s_{E_{k4}^{1}}, s_{E_{4l}^\nu} ]\right]\qquad  \qquad \qquad\qquad\qquad\qquad\qquad\qquad\\
=\quad [ L', s_{E_{ij}^\mu} [s_{E_{k4}^1}, s_{E_{4l}^\nu}]( s_{E_{ij}^\mu})^{-1}] 
\qquad [L'',s_{E_{k4}^1} [s_{E_{ij}^\mu}, s_{E_{4l}^\nu}]s_{E_{4l}^\nu} (s_{E_{k4}^1})^{-1}]\quad\\  \qquad
[L''', s_{E_{k4}^1} s_{E_{4l}^\nu}(s_{E_{k4}^1})^{-1} (s_{E_{4l}^\nu})^{-1} (s_{E_{k4}^1})^{-1} ]
\end{eqnarray*}
where $L'$ is the relator $[s_{E_{i4}^1}, s_{E_{4j}^\mu}] (s_{E_{ij}^{\mu}})^{-1}$, $L''$ is the relator 
$[s_{E_{ij}^\mu}, s_{E_{k4}^1}]$ and $L'''$ is the relator $s_{E_{k4}^1}[s_{E_{ij}^\mu}, s_{E_{4l}^\nu}](s_{E_{k4}^1})^{-1}$.  Thus $c(U_{ijkl}^{\mu,\nu})$ may be expressed as the product of 3 commutators, each involving a relator.  Similarly, we may derive expressions for $c(S_{ijkl}^{\mu,\nu})$, $c(T_{ijkl}^{\mu,\nu})$ expressing them as the product of 4 and 18 commutators respectively, all involving a relator.    During the expansion of $T_{ijkl}^{\mu,\nu}$ it was necessary to introduce the index 5, so from now on we work with $5 \times 5$ matrices.

Each $c(R_i)$ in (\ref{eqn:FreeIdCom}) may now be expressed as a product:
$c(R_i)= \prod_{j=1}^{m_i} [L_{ij},y_{ij}]^{\pm1}\!$,
so 
\begin{eqnarray*}(c(w_i) c(R_i) ^{\pm 1}  c(w_i)^{-1} )= c(w_i)\left(\prod_{j=1}^{m_i} \left[L_{ij},y_{ij}\right]^{\pm1}\right)^{\pm1}c(w_i)^{-1}\\=
\left(\prod_{j=1}^{m_i} \left[c(w_i)L_{ij}c(w_i)^{-1},\,\,c(w_i)y_{ij}c(w_i)^{-1}\right]^{\pm1}\right)^{\pm 1}
\end{eqnarray*}
and from (\ref{eqn:FreeIdCom})
\begin{eqnarray*}
[c(a'),c(b')] &=& \prod_{i=1}^m\left(\prod_{j=1}^{m_i} \left[c(w_i)L_{ij}c(w_i)^{-1},\,\,c(w_i)y_{ij}c(w_i)^{-1}\right]^{\pm1}\right)^{\pm1}=C_3.
\end{eqnarray*}
Here $C_3$ is the product of 2392 commutators, each involving a lengthy relator and word.  Conversely  $C_1=[K_a, s_a s_bs_a^{-1}]$ and $C_2=[s_b K_a s_a s_b^{-1},K_b]$ are much shorter. 

In order to apply the operation $W$ to both sides of  (\ref{eqn:LongWord}) we must rewrite the relator in each commutator in $C_1,C_2,C_3$ as a product of words of the form $(s_xs_ys_{xy}^{-1})^{\pm1}$.  We may use the inductive process from the proof of Lemma \ref{lemma:RelForm} to do this.  The number of terms of the form  $(s_xs_ys_{xy}^{-1})^{\pm1}$ needed to express each relator will be approximately the length of the relator, and each such term will add 2 terms to the 3-cycle produced by the application of $W$ (one coming from the relator and the other from its inverse).  

With the help of a computer we work through the left and right hand sides of (\ref{eqn:LongWord}) (which we will denote $u_1,u_2$  respectively) applying $W$.  Clearly this would produce millions of terms.  Whenever a new term was produced, we stored it in memory, and thereafter merely kept a running total of the coefficient on it.   Thus we obtained $W(u_2)$ as a 3-chain with 11123 terms, which between them involve 3691 distinct $5 \times 5$ matrices.  Subtracting from the much smaller $W(u_1)$ we apply Theorem \ref{BoundaryRelation} to get:
\[d'\left(W(v_1)-W(v_2)\right)=\{a,b\}.\]

We next ran an algorithm on $W(u_1)-W(u_2)$ searching for boundaries of 4-cycles which could be added or subtracted to shorten it.  Let $Z_1$ denote the result after adding and subtracting those boundaries.  This has merely 6844 distinct terms, involving between them 3265 matrices.   Thus we have attained the desired boundary relation:
\begin{equation}\label{eqn:boundaryZ1}
	d'Z_1=[a|b]-[b|a].
\end{equation}
\begin{remark}
The operator $d'$ was applied to $Z_1$  by computer, to independently verify this identity.
\end{remark}
\begin{remark}  
Clearly it is not possible to include in this article all the steps required to produce our expression for $Z_1$.  However, this expression is available from the authors and one can check that it satisfies Equation (\ref{eqn:boundaryZ1}) (see the end of \S\ref{section:basis}), that is, it is indeed a representative of the universal chain (cf.~Definition \ref{def:Z1}).
\end{remark}

\section{Computing the universal Borel class}\label{SectionTwo}
As an independent result we show how to compute the universal Borel class $b_n$ evaluated in the homology group $H_{2n+1}\left(GL(\mathbb{C})\right)$ by expanding the Karoubi-Hamida integral of \cite{Ha00}. This approach allows the calculation of the Borel regulator map for any number field after evaluation of the Hurewicz homomorphism. In particular for $n=1$ and $F$ a cyclotomic field, we can compute $V_1(F)$ (Theorem \ref{mothergoose}).


This section is organised as follows. We begin by recalling the definition of the Karoubi-Hamida integral (\S\ref{section:HamidaIntegral}). We expand this integral in an arbitrary odd dimension as an infinite series (\S\ref{section:PowerSeriesFormula}). The formula requires a certain matrix $A$ to have norm less than 1; in \S\ref{section:HomologicalTrick} we explain how to guarantee this condition. We then simplify the formula in dimension 3 (\S\ref{section:formula3}) in a way that can be implemented straightaway  in a computer algorithm. 


\subsection{Karoubi-Hamida's integral}\label{section:HamidaIntegral}

By a result of Hamida \cite{Ha00}, the universal Borel class $b_m$ has the following description as an integral of differential forms. Let $n=2m+1$ and let $X_0, \ldots, X_{n} \in GL_N(\mathbb{C})$, for some $N \ge 2n+1$ (cf.~Remark \ref{remark:SuslinStability}). Let $\Delta_n$ be the standard $n$-simplex
\begin{equation}\label{eqn:SimplexInRn1} 
   \Delta_n = \Big\{ (x_0, \ldots, x_{n}) \in \mathbb{R}^{n+1} \;\big|\; x_i \geq 0, \sum_i x_i = 1 \Big\}. 
\end{equation}
Define for every point $\mathbf{x} = (x_0, \ldots, x_{n}) \in \Delta_n$
\begin{equation}\label{eqn:nu}
	\nu(\mathbf{x}) = x_0X_0^*X_0 + \ldots + x_{n}X_{n}^*X_{n}\,,
\end{equation}
where $X^*$ denotes the conjugate transpose of $X$. Thus $\nu$ is a matrix of
0-forms (complex functions) on the $n$-manifold $\Delta_n$.  For any $\mathbf{x}$ and any non-zero vector $\vec{u}\in \mathbb{C}^N$, we have $\vec{u}^* \nu(\mathbf{x})\vec{u}$ a positive real number.
That is  $\nu(\mathbf{x})$ is positive definite hermitian and in particular invertible. 
Consider the matrix of differential $n$-forms
$(\nu^{-1}d\nu)^{n}$, where $\nu^{-1}$ denotes matrix inversion, 
$d$ is the exterior derivative applied to each entry of $\nu$ and we multiply individual differential forms using the wedge product. Define
\begin{equation}
  \label{eq:HamidaIntegral}
    \varphi(X_0, X_1, \ldots, X_n) = \Tr \int_{\Delta^{n}} (\nu^{-1} d\nu)^{n}\,,
\end{equation}
the trace of a matrix of integrals of differential $n$-forms.

\begin{thm}[Hamida \cite{Ha00}] \label{thm:Hamida}
Let $m \ge 1$ and $n = 2m+1$.
The map defined on tuples in the homogeneous bar resolution, sending
\begin{equation}\label{eqn:HamidaCocycle}
    (X_0, \ldots, X_{n}) \,\mapsto \,\frac{(-1)^{m + 1}}{2^{3m+1}(\pi i)^{m}}
    \:\varphi(X^\ast_0, X^\ast_1, \ldots, X^\ast_{n})
\end{equation}
 is a cocycle representing the universal Borel class $b_m \colon H_{2m+1}\left( GL(\mathbb{C}) \right) \to \mathbb{R}(m)$.
\end{thm}
\begin{remark}
Note that our convention $X_i^\ast{}X_i$ is different from \cite{Ha00}: the definition of $\nu(\mathbf{x})$  there is $\sum_i x_i X_iX_i^\ast$. 
\end{remark}

\begin{remark}\label{glutus:maximus}
In fact $\varphi$ is an alternating map.
\end{remark}

\begin{remark}\label{rmk:HomoUni}
The cocycle (\ref{eqn:HamidaCocycle}) is homogeneous and unitarily normalized \cite{Ha00}, that is,
\begin{eqnarray}
  \varphi(X_0g, \ldots, X_ng) &=& \varphi(X_0, \ldots, X_n) \quad \textup{for all } g \in GL_N(\mathbb{C}),\label{homogeneous}\\
   \varphi(u_0X_0, \ldots, u_nX_n) &=& \varphi(X_0, \ldots, X_n) \quad \textup{for all } u_i \in U_N(\mathbb{C}).\label{unitary}
\end{eqnarray}
In particular, we can assume $X_n=1$ by (\ref{homogeneous}), and all the $X_i$ to be positive definite hermitian matrices by (\ref{unitary}), via the polar decomposition: every invertible matrix $X$ can be written as $X=UP$ where $U$ is unitary and $P$ is positive definite hermitian. Indeed, $X^*X=P^*P=P^2$; compare with (\ref{eqn:nu}).
\end{remark}

\subsection{The infinite series formula}\label{section:PowerSeriesFormula}
Our goal is to make the computation of the Karoubi-Hamida integral (\ref{eq:HamidaIntegral}) explicit. Namely, we will transform the integral into an infinite series whose value we can arbitrarily approximate. 

\paragraph{Step 1} \textit{Express the integrand in terms of $n$ coordinates rather than $n+1$}

\smallskip 

\noindent We have a homeomorphism from the $n$-simplex in ${\mathbb R}^{n}$
\[  
	\Delta^{n} = \{ (y_{1}, y_{2}, \ldots , y_{n})  \  |  \  y_{i} \geq 0 \  {\rm for \ all } \ i \ ,  \sum_{j=1}^{n} \ y_{j} \leq 1 \}    
\]
to $\Delta_n \subset{\mathbb R}^{n+1 }$ 
given by the map 
\[  
	(y_{1}, \ldots , y_{n}) \mapsto ( 1 -  \sum_{j=1}^{n} \ y_{j} , y_{1}, y_{2}, \ldots , y_{n})  .  
\]
Therefore, in terms of the $y$-coordinates, we have
\begin{eqnarray}\label{eqn:nu1}
 	\nu & = &  X_{0}^{*}X_{0} +  y_{1}( X_{1}^{*}X_{1} -  X_{0}^{*}X_{0}) + \ldots   + y_{n} (X_{n}^{*}X_{n} -  X_{0}^{*}X_{0}). 
\end{eqnarray}

\paragraph{Step 2} \textit{Change of variables} \smallskip

\noindent Next, we perform a change of variables by means of the map
$$
	 T \colon [0, 1] \times \Delta^{n-1}  \longrightarrow  \Delta^{n}  
$$
given by
\[ T(t, s_{1}, s_{2}, \ldots , s_{n-1}) =  (  s_{1}t, s_{2}t, \ldots, s_{n-1}t, 1-t). \]

For each fixed non-zero value of $t$ the corresponding horizontal $(n-1)$-simplex is mapped diffeomorphically onto its image,
while $\{ 0 \} \times \Delta^{n-1}$ is collapsed to the vertex $(0,0, \ldots , 0, 1)$. 
The Jacobian of $T$ is equal to $(-1)^{n} t^{n-1}$.
Therefore when $t \not= 0$ we have $J(T) > 0$ if $n$ is even and $J(T) < 0$ if $n$ is odd. For any compact $n$-manifold with boundary $M \subset [0,1] \times \Delta^{n-1}$ having image $T(M)$ in the $n$-simplex and continuous map $f \colon T(M) \to \mathbb{R}$ we have the substitution rule \cite[p.~28]{BT}
\begin{multline*}
 \int_{M}  \  (f \circ T)(t, s_{1}, \ldots , s_{n-1}) \, t^{n-1} \, dt \,ds_{1} \ldots  d\,s_{n-1} \\
 =  \int_{T(M)}  \ f(y_{1}, \ldots , y_{n}) dy_{1} \ldots dy_{n}\,,   
\end{multline*}
since the absolute value of the Jacobian is just $t^{n-1}$.

We shall be interested in the integral  ${\rm Tr} \int_{\Delta^{n}} ((\nu)^{-1} d\nu )^{n} $ which may be written as the limit of integrals over manifolds of the form $T(M)$ as they tend towards the $n$-simplex. (For example, when $M_\alpha=[\alpha,1]\times \Delta^{n-1}$ and $\alpha \to 0^+$.) Therefore we can compute this integral as a limit of corresponding integrals over $M \subset [0,1] \times \Delta^{n-1}$, provided that this limit exists. However the integral over $M$, involving the Jacobian of $T$, is merely the integral over $M$ where $\nu \circ T$ is $\nu$ written in terms of $t, s_{1}, \ldots , s_{n-1}$ and $d\nu$ is also computed in these coordinates. Thus from (\ref{eqn:nu1}) and the definition of $T$ we have
\begin{align*}
 	\nu \circ T & = X_{0}^{*}X_{0} + \sum_{j=1}^{n-1} t s_j (X_j^\ast X_j -  X_0^\ast X_0) + (1-t) (X_n^\ast X_n - X_0^\ast X_0)\\
 		           & = X_n^\ast X_n + t A(s_{1}, \ldots , s_{n-1}),
\end{align*}
where
\begin{align}\label{eqn:MatrixA}
	A(s_{1}, s_{2}, \ldots , s_{n-1})  =  (X_{0}^{*}X_{0} - X_n^\ast X_n) +
\sum_{j=1}^{n-1} \ s_{j}( X_{j}^{*}X_{j} -  X_{0}^{*}X_{0} ) . 
\end{align}
Therefore we shall compute
$$
	\Tr \int_M \left((\nu')^{-1} d\nu' \right)^{n} (-1)^{n}t^{n-1}
$$	
where $\nu'=\nu \circ T$, that is,
$$
	\nu'(t,s_1,\ldots,s_n) = X_n^\ast X_n + t A(s_{1}, \ldots , s_{n-1})
$$
for all $(t,s_1,\ldots,s_{n-1}) \in M$ where $M \subset [0,1] \times \Delta^{n-1}$ is an arbitrary compact $n$-manifold with boundary.

\paragraph{Step 3} \textit{Assume $X_n=1$} \smallskip

\noindent Taking $g = X_n^{-1}$ in  (\ref{homogeneous}) we may from now on assume $X_n = 1$ (we write 1 for the identity matrix whenever there is no possibility of confusion), and hence 
\begin{eqnarray}
	\nu(t, s_1, \ldots, s_{n-1}) &=& 1 + t A(s_1,\ldots, s_{n-1}), \nonumber\\
	A(s_1,\ldots, s_{n-1}) &=& (X_{0}^{*}X_{0} - 1) + \sum_{j=1}^{n-1} \ s_{j}( X_{j}^{*}X_{j} -  X_{0}^{*}X_{0} ) \label{eqn:MatrixASimplified}. 
\end{eqnarray}
\begin{remark}
Alternatively, Proposition \ref{prop:BoundA} allows us to always assume $X_n=1$.
\end{remark}
Therefore $d\nu = dt\, A + t\, dA$, with $dA = \sum_{j=1}^{n-1} \ ds_{j}( X_{j}^{*}X_{j} -  X_{0}^{*}X_{0} )$ and as $M$ varies, we must examine the integral
\begin{equation}\label{eqn:UnaIntegral} 
	\Tr \int_{M} t^{n-1} \left(\left( 1 + t A\right)^{-1} \left(dt\, A + t\, dA\right)\right)^{n} .
\end{equation}

\paragraph{Step 4} \textit{Commuting factors and cyclic permutations} \smallskip

\noindent We have
\begin{multline*}
 \Tr \int_{M} t^{n-1}\left(\left( 1 + t A\right)^{-1} \left(dt\, A + t\, dA\right)\right)^{n}\\  =
\Tr \int_{M} t^{n-1} \left(\left( 1 + t A\right)^{-1} dt\, A +\left( 1 + t A\right)^{-1} t\, dA\right)^{n}.
\end{multline*} 
Write $Y = \left( 1 + t A\right)^{-1} dt\, A$ and $Z = \left( 1 + t A\right)^{-1} t\, dA$.  The 1-form $dt$ commutes with the 0-forms $\left( 1 + t A\right)^{-1}$ and $t$, and anticommutes with the 1-form $dA$.  In the expansion of $(Y+Z)^n$, any monomial involving  more than one $Y$ will vanish, as it contains $dt \, dt =0$. In addition, $Z^n=0$ as it is an $n$-form on $n-1$ variables $s_1, \ldots, s_{n-1}$. Consequently,
\begin{align}\label{eqn:CyclicSum}
	 \Tr \int_{M} \left(Y+Z\right)^n = \sum_{j=0}^{n-1} \Tr \int_{M} Z^jYZ^{n-1-j}.
\end{align}
Next we observe that if $W_{1}, \ldots ,  W_{n}$ are $m \times m$ matrix-valued functions then
\begin{equation}\label{eqn:CyclicPermutation}
	{\rm Tr}\left(W_{1} dx_{1} W_{2} dx_{2} \ldots W_{n} dx_{n}\right)  =  (-1)^{n-1} {\rm Tr}\left( W_{2} dx_{2} \ldots W_{n} dx_{n} W_{1} dx_{1}\right) 
\end{equation}
because the trace of a product of $n$ matrices is invariant under cyclic permutations but the $n$-form $dx_{1} \ldots dx_{n}$ changes by the sign of the $n$-cycle, which is $(-1)^{n-1}$. Accordingly each integral in the sum (\ref{eqn:CyclicSum}) equals $(-1)^{n-1}$ the next one. This implies that if $n$ is even the sum is zero, that is,
\[    
	\varphi(X_0, X_1, \ldots, X_{2n}) = \Tr \int_{\Delta^{2n}} (\nu^{-1} d\nu)^{2n} = 0. 
\]  
For the rest of this section we shall assume that $n \ge 3$ is odd (for the case $n=1$ see Remark \ref{remark:Casen1}). In this case all the summands in (\ref{eqn:CyclicSum}) are equal and consequently
\begin{eqnarray*}
\lefteqn{{\rm Tr} \int_{M} t^{n-1} \left(\left( 1 + t A\right)^{-1} \left(dt\, A + t\, dA\right)\right)^{n}} \\
& = & n \ {\rm Tr} \int_{M} t^{n-1} \left(1 +  t A \right)^{-1}  dt\, A \left(\left(1 + t A\right)^{-1} t\, dA \right)^{n-1}  \\
 & = & n \ {\rm Tr}\int_{M}  t^{2n-2} \left(1 + t A \right)^{-1} dt\, A  \left(\left(1 + t A \right)^{-1}  dA \right)^{n-1}.\label{eqn:nisone}
\end{eqnarray*}
Now $A$ commutes with $1+tA$ and hence with its inverse. Thus the previous integral equals
\begin{multline}
 n \ {\rm Tr}\int_{M}  t^{2n-2} dt\, A  \left(1 + t A \right)^{-1}  \left(\left(1 + t A \right)^{-1}  dA \right)^{n-1}\\
=n \ {\rm Tr}\int_{M}   t^{2n-2} dt\, A  \left(1 + t A \right)^{-2} \left(dA \left(1 + t A \right)^{-1}\right)^{n-2}dA. \label{eqn:LastIntegral}
\end{multline}

\paragraph{Step 5} \textit{Invert $1 + t A$} \medskip

\noindent Recall the geometric series formula for a matrix $A$ \cite[5.6.16]{Horn:MatrixAnalysis}: if $\|\cdot\|$ is a matrix norm and $\|A\| < 1$ then $1 - A$ is invertible and $\sum_{k=0}^\infty A^k = (1 - A)^{-1}$ with respect to $\|\cdot\|$. (A \emph{matrix norm} on $M_N(\mathbb{C})$ is a vector norm which satisfies $\|X Y\| \le \|X\| \|Y\|$.) In order to invert $1+tA$, we assume from now on that $\|A\| < 1$ through the domain of integration. There is a justification, explained in \S\ref{section:HomologicalTrick}, which allows us to do so.
\begin{remark}\label{rmk:NormA}
By $\|A\|<1$ we formally mean $\|A(s_1,\ldots,s_{n-1})\|<1$ for all $(s_1,\ldots,s_{n-1}) \in \Delta^{n-1}$. Equivalently, we may define $\|A\|$ as the maximum of this function over the compact set $\Delta^{n-1}$.
\end{remark}
\noindent{}The geometric series for $-tA$ gives
\[
	(1+t A)^{-1} = \sum_{k=0}^\infty (-tA)^k = \sum_{k=0}^\infty (-1)^k t^k A^k\, ,
\]
and it follows that
\[
	(1+t A)^{-2} = \sum_{k=0}^\infty (-1)^k (k+1) t^kA^{k}.
\]
Hence under the assumption $\|A\| < 1$ we may express the integral (\ref{eqn:UnaIntegral}) as a convergent infinite series in the following manner.
\begin{eqnarray*}
	 \lefteqn{ {\rm Tr} \int_M t^{n-1} \left(\left( 1 + t A\right)^{-1} \left(dt A + t dA\right)\right)^{n} } \\
& \stackrel{(\ref{eqn:LastIntegral} )}{=} & n \ {\rm Tr}\int_M t^{2n-2} dt\, A  (1 + t A )^{-2} dA (1 + t A )^{-1}  dA \ldots  (1 + t A )^{-1}  dA  \\
& = &  n  \ {\rm Tr}\int_M t^{2n-2} dt\, A \sum_{ m_{i} \ge 0 } (-1)^{m_1} (m_1+1) (tA)^{m_1} dA \ldots (-1)^{m_{n-1}} (tA)^{m_{n-1}} dA   \\
& = &   n \ {\rm Tr}\int_M  \sum_{ m_{i} \ge 0   } \ (-1)^{\underline{m}}\,
(m_1+1) \, t^{\underline{m}+2n-2} dt\, A^{m_1+1}\,dA\, A^{m_2} dA \ldots A^{m_{n-1}} dA \\
& = &  n \ {\rm Tr}\int_M  \sum_{ m_{i} \ge 0   } \ (-1)^{\underline{m}-1}\,
m_1 \, t^{\underline{m}+2n-3} dt\, A^{m_1}\,dA\, A^{m_2} dA \ldots A^{m_{n-1}} dA\, ,
\end{eqnarray*}
where $\underline{m}$ is our short notation for $\underline{m} (m_1,  \ldots , m_{n-1}) = m_{1} + \ldots + m_{n-1}$.

For $0 < \alpha < 1$ write $M_\alpha = [\alpha,1] \times \Delta^{n-1}$. The conditional convergence and Fubini theorems guarantee that the original integral over $\Delta^n$ equals
\begin{eqnarray*}
	\lefteqn{\lim_{\alpha\to 0^+} \int_{M_\alpha} \sum_{ m_{i} \ge 0   } \ (-1)^{\underline{m}-1}\,
m_1\,  t^{\underline{m}+2n-3} dt\, A^{m_1} dA\, A^{m_2} dA \ldots A^{m_{n-1}} dA}\\
	& = & \sum_{ m_{i} \ge 0   } \frac{(-1)^{\underline{m}-1}\, m_1}{\underline{m}+2n-2} \int_{\Delta^{n-1}} A^{m_1}\,dA\, A^{m_2} dA \ldots A^{m_{n-1}} dA\,,
\end{eqnarray*}
since 
$$
	\lim_{\alpha\to 0^+} \int_\alpha^1 t^{\underline{m}+2n-3}\, dt = \frac{1}{\underline{m}+2n-2}.
$$
%
%

\paragraph{Step 6} \textit{Expand the powers of $A$} \medskip

\noindent{}
Let us write
$$
	A  =  U_0 + \sum_{j=1}^{n-1} U_j s_j
$$
where $U_0 = X_0^\ast X_0 - 1$, $U_j = X_j^\ast X_j -  X_0^\ast X_0$ for $1 \le j \le n-1$ as in Eq.~(\ref{eqn:MatrixASimplified}). 
\begin{remark}
If $X_n \neq 1$ then we have $U_0 = X_n\left(X_0^\ast X_0 - 1\right)X_n^{-1}$ and, for $1 \le j \le n-1$, $U_j =X_n\left(X_j^\ast X_j -  X_0^\ast X_0\right)X_n^{-1}$.
\end{remark}

Then 
$$
	dA  =  \sum_{j=1}^{n-1} U_j ds_j
$$
and we can write
\begin{eqnarray*}
	A^{m_1}\, dA \ldots A^{m_{n-1}}\, dA &=& \prod_{i=1}^{n-1} \left( U_0 + \sum_{j=1}^{n-1} U_j s_j \right)^{m_i}
		\left(\sum_{j=1}^{n-1} U_j ds_j \right)\\
		&=& \sum_{|\mathbf{l}| \le |\mathbf{m}|} U(l_1,l_2,\ldots, l_{n-1})\,  s_1^{l_1}s_2^{l_2}\ldots s_{n-1}^{l_{n-1}} ds_1 ds_2 \ldots ds_{n-1}.
\end{eqnarray*}
Here the sum is over all nonnegative integer vectors $\mathbf{l} = (l_1,\ldots, l_{n-1})$ such that $|\mathbf{l}| = \sum_i l_i \le \sum_i m_i = |\mathbf{m}|$, and $U(l_1,\ldots, l_{n-1})$ is defined as the matrix coefficient of $s_1^{l_1}\ldots s_{n-1}^{l_{n-1}}$ in the expansion of the previous line. Since this matrix depends on both $\mathbf{l}$ and $\mathbf{m}=(m_1,\ldots,m_{n-1})$, we will sometimes write $U(\mathbf{m},\mathbf{l})$. It can be described more explicitly:
\begin{lem}\label{lem:U}
The matrix $U(\mathbf{m},\mathbf{l})$ equals the sum of all the matrices of the form
$$
	s\cdot V^1_1\ldots V^{1}_{m_1}W_1 V^{2}_{1} \ldots V^{2}_{m_2}W_2 \ldots V^{n-1}_{1}\ldots V^{n-1}_{m_{n-1}}W_{n-1}
$$
where 
\begin{enumerate}
	\item $V^i_j \in \{ U_k \colon 0\le k \le n-1\}$ for all $i,j$;
	\item $(W_1, \ldots, W_{n-1})$ is a permutation of $(U_1,\ldots,U_{{n-1}})$ of signature $s \in \{\pm 1\}$;
	\item 
$l_k$ is the cardinality of the set $\{ (i,j) \,|\,\ V^i_j = U_k\}$, for each $k$.
\end{enumerate}
\end{lem}

\paragraph{Step 7} \textit{Remove the remaining integral} \medskip

\noindent{}Since the matrices $U(\mathbf{m},\mathbf{l})$ are constant with respect to the variables $s_i$, we have
\begin{eqnarray*}
	 \lefteqn{\int_{\Delta^{n-1}} A^{m_1}\,dA\, A^{m_2} dA \ldots A^{m_{n-1}} dA}  \\
	& = &  \sum_{|\mathbf{l}| \le |\mathbf{m}|} U(\mathbf{m},\mathbf{l})  \int_{\Delta^{n-1}} s_1^{l_1}s_2^{l_2}\ldots s_{n-1}^{l_{n-1}} ds_1 ds_2 \ldots ds_{n-1}. 
\end{eqnarray*}

\begin{lem}\label{lem:fact}
$$
	\int_{\Delta^{n-1}} s_1^{l_1} \ldots s_{n-1}^{l_{n-1}} ds_1 \ldots ds_{n-1} = \frac{l_1! l_2!\ldots l_{n-1}!}{(l_1+l_2+\ldots +l_{n-1}+n-1)!}\,.
$$
\end{lem}
\noindent{}To ease notation we write the number above as $\textrm{fact}(l_1,\ldots,l_{n-1})$ or $\textrm{fact}(\mathbf{l})$. We leave the proof of the lemma as a multivariable calculus exercise.

\smallskip

On the whole we have proven the following.
\begin{thm}\label{thm:MainThm}
Let $n \ge 3$ odd and  $X_0, \ldots, X_n \in GL_N(\mathbb{C})$. Let 
$$
	  A = U_0 + \sum_{j=1}^{n-1} U_j s_j 
$$ 
where $U_0 = X^\ast_n\left(X_0^\ast X_0 - 1\right)X_n^{-1}$ and $U_j =X_n\left(X_j^\ast X_j -  X_0^\ast X_0\right)X_n^{-1}$ for $1 \le j \le n-1$. Suppose that  $\norm{A} < 1$.
Then Hamida's function $\varphi(X_0, \ldots, X_n)$ equals the limit of the convergent series
 \begin{equation*}
n \sum_{ |\mathbf{m}|=0 }^\infty \frac{(-1)^{|\mathbf{m}|-1}}{|\mathbf{m}|+2n-2} \, m_1
\sum_{|\mathbf{l}| \le |\mathbf{m}|} \textup{fact}(\mathbf{l}) \, \Tr\left( U(\mathbf{m},\mathbf{l}) \right)\, 
\end{equation*}
where
\begin{enumerate}
	\item the outer sum is over all nonnegative integer vectors $\mathbf{m}=(m_1,\ldots,m_{n-1})$, that is, the limit when $k \to \infty$ of the finite sums over $|\mathbf{m}|\le k$;
	\item the (finite) inner sum is over all nonnegative integer vectors $\mathbf{l}=(l_1,\ldots,l_{n-1})$ such that $|\mathbf{l}| \le |\mathbf{m}|$;
	\item $\textup{fact}(\mathbf{l})$ is as defined in Lemma \ref{lem:fact};
	\item $U(\mathbf{m},\mathbf{l})$ is as defined in Lemma \ref{lem:U}.
\end{enumerate}
\end{thm}

\begin{remark}[Case $n=1$] \label{remark:Casen1}
In this case we only need Steps 1 to 3 and the geometric series formula to invert $1+tA$. The  map $T$ can be taken to be the identity map,  $M = [0,1]$ and note that the matrix $A$ is constant. From (\ref{eqn:UnaIntegral}) and assuming $\|A\| < 1$,
\begin{align*}
	\Tr \int_{ \Delta^{1}}   ( \nu^{-1} d \nu)  & =   \Tr \int_{0}^{1}  (1 + t A)^{-1} A\, dt \\
&  = \Tr  \, \int_{0}^{1} \sum_{m \geq 0}  (-1)^{m} t^{m} A^{m} A\, dt  \\
&  = \Tr  \, \sum_{m \geq 0}  (-1)^{m} A^{m+1} \int_{0}^{1} t^{m} dt  \\
& =    \Tr \left( \log\left(1 + A\right)\right) .
\end{align*}
\end{remark}

In Section \ref{section:formula3} we will consider the case $n=3$ in more detail (the relevant case for the computation of $V_1(F)$). Before that, we explain how to ensure the condition $\norm{A} < 1$.

\subsection{Controlling the norm}\label{section:HomologicalTrick}
In order to use the geometric series to invert $1+tA$ (Step 5 in \S\ref{section:PowerSeriesFormula}) we need $\|A\| < 1$ for any matrix norm $\|\cdot\|$. We ensure this condition for the spectral norm by using a homological trick. First we briefly discuss matrix norms.

Let $\norm{\cdot}$ a \emph{matrix norm}, that is, a vector norm in $M_N(\mathbb{C})$ which satisfies $\|X Y\| \le \|X\| \|Y\|$. Our main example will be the \emph{spectral
norm}\label{def:SpectralNorm} \cite[5.6.6]{Horn:MatrixAnalysis}
$$
    \| X \|_{2} = \max\left\{ \sqrt{\lambda} \,:\,
    \lambda \textrm{ is an eigenvalue of } X^*X \right\}.
$$
(Note that $X^\ast X$ is positive semidefinite (hermitian) and hence all eigenvalues are real and nonnegative.)
This norm satisfies
\begin{itemize}
	\item[(i)] $\norm{X^\ast} = \norm{X}$ for all $X \in M_N(\mathbb{C})$; 
	\item[(ii)] $\norm{X^\ast X} = \norm{X}^2$ for all $X \in M_N(\mathbb{C})$;
	\item[(iii)] if $X$ is hermitian then $\norm{X} = \max\left\{ |\lambda| \,:\, \lambda \textrm{ is an eigenvalue of } X \right\}$;
	\item[(iv)] if $X$ is positive semidefinite then $\norm{X} = \lambda_\text{max}(X)$ the maximum eigenvalue.
\end{itemize}
If $X$ is hermitian, all its eigenvalues are real and we will write them in increasing order as
$$ \lambda_{\rm min}(X) = \lambda_1(X) \le \lambda_2(X) \le \cdots
\le \lambda_N(X) = \lambda_{\rm max}(X).$$
\begin{lem}\label{lem:Weyl}
Let $X, Y$ be hermitian matrices. Then $\lambda_{\rm max}(X+Y) \le \lambda_{\rm max}(X) + \lambda_{\rm max}(Y)$ and 
 $\lambda_{\rm min}(X+Y) \ge \lambda_{\rm min}(X) + \lambda_{\rm min}(Y)$.
\end{lem}
\noindent{}This Lemma follows from Theorem 4.3.1 in \cite{Horn:MatrixAnalysis}.

If $X$ is a matrix of functions over a set $\Delta \subset \mathbb{R}^n$ such that $X(\mathbf{s})$ is hermitian for each $\mathbf{s} \in \Delta$, we define
\begin{eqnarray*}
	\lambda_{\rm max}(X) &=& \sup_{\mathbf{s}\in \Delta} \lambda_{\rm max}(X(\mathbf{s})),\\
	\lambda_{\rm min}(X) &=& \inf_{\mathbf{s}\in \Delta} \lambda_{\rm min}(X(\mathbf{s})),\\
	\norm{X} &=& \sup_{\mathbf{s}\in \Delta} \norm{X(\mathbf{s})}.\\
\end{eqnarray*}
This definition is consistent with our previous notation (Remark \ref{rmk:NormA}).
Note that if $X(\mathbf{s})$ is positive definite for all $\mathbf{s}$ then $\norm{X}=\lambda_{\rm max}(X)$.

Let $X_0, \ldots, X_{n} \in GL_N(\mathbb{C})$ and consider for each $\mathbf{s}=(s_1,\ldots, s_{n}) \in \Delta^{n}$
\begin{eqnarray*}
	A(\mathbf{s}) &=&  X_0^\ast X_0 - I + \sum_{\substack{j=1}}^n  s_j \left(X_j^\ast X_j - X_0^\ast X_0 \right).
\end{eqnarray*}
(This is the matrix $A$ associated to the tuple $(X_0,X_1,\ldots,X_n,I)$ as in \S\ref{section:PowerSeriesFormula}.)
We may write $A(\mathbf{s}) = H(\mathbf{s}) - I$ where 
$$
	H(\mathbf{s}) = X_0^\ast X_0 \Big( 1 - \sum_{\substack{j=1}}^n s_j \Big) + \sum_{\substack{j=1}}^n s_j X_j^\ast X_j 
$$
is a positive definite hermitian matrix (a positive linear combination of positive definite hermitian matrices). 
Define
\begin{eqnarray*}
	\lambda_{\textup{max}} &=& \max_{0\le j \le n} \lambda_{\textup{max}}(X_j^\ast X_j) = \max_{i,j} \lambda_{i}(X_j^\ast X_j) \quad \text{and}\\
	\lambda_{\textup{min}} &=& \min_{0\le j \le n} \lambda_{\rm min}(X_j^\ast X_j) = \min_{i,j} \lambda_{i}(X_j^\ast X_j).
\end{eqnarray*}
\begin{lem}\label{lem:Bounds}
We have
\begin{itemize}
	\item[(i)] $\norm{H} = \lambda_{\textup{max}}(H) = \lambda_{\textup{max}}$, $\lambda_{\textup{min}}(H) = \lambda_{\textup{min}}$;
	\item[(ii)] $\norm{A} = \max \{ |\lambda_{\rm max} - 1|, |\lambda_{\rm min} - 1|\}$;
	\item[(iii)] $\norm{A} < 1$ if and only if $\lambda_\textup{max} < 2$.
\end{itemize}
\end{lem}
\begin{proof}
Using Lemma \ref{lem:Weyl} we have
\[
\begin{array}{l}
	\lambda_{\textup{max}}\left(H(\mathbf{s})\right) \le \Big( 1 - \sum_{\substack{j=1}}^n s_j \Big)\lambda_{\textup{max}}(X_0^\ast X_0)  + \sum_{\substack{j=1}}^n s_j\, \lambda_{\textup{max}}(X_j^\ast X_j) \le \lambda_{\textup{max}}  \ \text{ and}\\
	\lambda_{\textup{min}}\left(H(\mathbf{s})\right) \ge \Big( 1 - \sum_{\substack{j=1}}^n s_j \Big)\lambda_{\textup{min}}(X_0^\ast X_0)  + \sum_{\substack{j=1}}^n s_j \, \lambda_{\textup{min}}(X_j^\ast X_j) \ge \lambda_{\textup{min}} 
\end{array}
\]
for all $\mathbf{s} = (s_1,\ldots,s_{n-1}) \in \Delta^{n-1}$. This proves part (i).\\
For part (ii), we have 
\[
	\norm{A} = \sup_\mathbf{s}\norm{H(\mathbf{s}) - I} = \sup_\mathbf{s} \max_i |\lambda_i (H(\mathbf{s}) - I)| = \sup_{\mathbf{s}, i}|\lambda_i (H(\mathbf{s})) - 1|.
\]
The supremum of the distances between a point and the points of a bounded subset of $\mathbb{R}$ equals the distance to either the supremum or the infimum of the set,
\[
	\norm{A} = \max \{ |\sup_{\mathbf{s},i} (\lambda_i (H(\mathbf{s}))) - 1|, |\inf_{\mathbf{s},i} (\lambda_i (H(\mathbf{s}))) - 1| \} = \max \{ |\lambda_{\textup{max}} - 1|, |\lambda_{\textup{min}} - 1| \},\\
\]
using part (i).\\
Part (iii) follows from (ii) observing that $0 < \lambda_\textup{min} \le \lambda_\textup{max}$.
\end{proof}
We now explain how to guarantee the condition $\norm{A} < 1$ by rescaling the matrices $X_i$. Let $\mu > 0$. The boundary of the tuple $(X_0, \ldots, X_{n}, \mu^{-1}I)$ is
$$
	\sum_{i=0}^{n} (-1)^i (X_0, \ldots, \widehat{X_i}, \ldots, X_{n}, \mu^{-1}I) + (-1)^{n+1} (X_0, \ldots, X_{n}).
$$
Since Hamida's function $\varphi$ is a cocycle, it vanishes on boundaries and thus 
\begin{eqnarray*}
	\varphi(X_0, \ldots,  X_{n}) &=& \sum_{i=0}^n (-1)^{n+i} \varphi(X_0, \ldots, \widehat{X_i}, \ldots, X_{n}, \mu^{-1}I)
\\ &=& \sum_{i=0}^n (-1)^{n+i} \varphi(\mu{}X_0, \ldots, \widehat{\mu{}X_i}, \ldots, \mu{}X_{n}, I)\,,
\end{eqnarray*}
the last equality coming from multiplying by a diagonal matrix with $\mu$ in the diagonal (Eq.(\ref{homogeneous}) in Remark \ref{rmk:HomoUni}).

We now prove that if $\mu$ is small enough, the matrix $A$ associated to any of the tuples on the right-hand side satisfies $\|A\|<1$ for the spectral norm. Hence the value of Hamida's function $\varphi$ at $(X_0, \ldots,  X_{n})$ can be computed as the alternating sum of the values at tuples satisfying the hypotheses of Theorem \ref{thm:MainThm}.
\begin{prop}\label{prop:BoundA}
Let $X_0, \ldots, X_{n} \in GL_N(\mathbb{C})$ and $\mu > 0$. Then
\begin{eqnarray*}
	\varphi(X_0, \ldots,  X_{n}) &=& \sum_{i=0}^n (-1)^{n+i} \varphi(\mu{}X_0, \ldots, \widehat{\mu{}X_i}, \ldots, \mu{}X_{n}, I)\,.
\end{eqnarray*}
Let $\lambda_{\textup{max}} = \max_{0\le j \le n} \lambda_{\textup{max}}(X_j^\ast X_j)$. If $0 < \mu < \sqrt{\frac{2}{\lambda_\textup{max}}}$ then for each tuple on the right-hand side, the associated matrix $A$ satisfies $\norm{A}<1$.
\end{prop}
\begin{proof}
We are left with the proof of the second statement. 
Let $A$ be the matrix associated to $(\mu{}X_0, \ldots, \widehat{\mu{}X_i}, \ldots, \mu{}X_{n}, I)$ for some $0 \le i \le n$. That is,
\begin{eqnarray*}
	A(\mathbf{s}) &=& \mu^2 X_0^\ast X_0 - I + \mu^2\sum_{\substack{j=1}}^{i-1}  s_j \left(X_j^\ast X_j - X_0^\ast X_0 \right) + \mu^2\sum_{\substack{j=i+1}}^n  s_{j-1} \left(X_j^\ast X_j - X_0^\ast X_0 \right),
\end{eqnarray*}
for each $\mathbf{s}=(s_1,\ldots,s_{n-1}) \in \Delta^{n-1}$. Let us write $A(\mathbf{s}) = H(\mathbf{s}) - I$. By Lemma \ref{lem:Bounds}(i) we have
\[
	\lambda_{\textup{max}}(H) = \max_{j\neq i} \lambda_\textup{max}(\mu^2 X_j^\ast X_j) \le \mu^2 \lambda_\textup{max} < 2.
\]
The result follows now from Lemma \ref{lem:Bounds}(iii). 
\end{proof}

For computational purposes (cf.~\S\ref{section:ComputerAlgorithm}) we will be interested in minimazing $\norm{A}$. We record the relevant result here, followed by a remark.
\begin{lem} \label{optimal}
Let $X_0, \ldots, X_{n} \in GL_N(\mathbb{C})$, $\lambda_{\textup{max}} = \max_{0\le j \le n} \lambda_{\textup{max}}(X_j^\ast X_j)$, $\lambda_{\textup{min}} = \min_{0\le j \le n} \lambda_{\textup{min}}(X_j^\ast X_j)$. Given $\mu > 0$ write $A_\mu$ for the matrix associated to the tuple $(\mu X_0, \ldots, \mu X_n)$. Then the function $\mu \mapsto \norm{A_\mu}$ reaches a minimum value $\frac{\lambda_\textup{max}-\lambda_\textup{min}}{\lambda_\textup{max}+\lambda_\textup{min}}$ at $\mu = \sqrt{\frac{2}{\lambda_\textup{max}+\lambda_\textup{min}}}$.
\end{lem}
\begin{proof}
By Lemma \ref{lem:Bounds}(ii)
\[
	\norm{A_\mu} = \max \{ |\mu^2\lambda_\textup{max} - 1|, |\mu^2\lambda_\textup{min} - 1|\}.
\]
The maximum of the distances to 1 reaches a minimum when both points are equidistant
\begin{eqnarray*}
	\mu^2\lambda_\textup{max} - 1 = 1 - \mu^2\lambda_\textup{min} &\: \Leftrightarrow \:& \mu = \sqrt{\frac{2}{\lambda_\textup{max}+\lambda_\textup{min}}}\, .
\end{eqnarray*}
For this value of $\mu$
\[
	\norm{A_\mu} = \mu^2\lambda_\textup{max} - 1 = \frac{2\lambda_\textup{max}}{\lambda_\textup{max}+\lambda_\textup{min}}-1 = \frac{\lambda_\textup{max}-\lambda_\textup{min}}{\lambda_\textup{max}+\lambda_\textup{min}}. \qedhere
\]
\end{proof}
\begin{remark}\label{remark:OptimalMu}
Note that $\mu = \sqrt{\frac{2}{\lambda_\textup{max}+\lambda_\textup{min}}} \le \sqrt{\frac{2}{\lambda_\textup{max}}}$ and hence Proposition \ref{prop:BoundA} applies for this choice of scaling factor $\mu$. Moreover, for all tuples on the right-hand side except at most two, the associated matrix $A$ will reach its minimum norm.
\end{remark}

\subsection{The infinite series for $n=3$}\label{section:formula3}
In this section we simplify and rearrange the formula in Theorem \ref{thm:MainThm} for $n=3$; this is the case relevant for the computation of $V_1(F)$. The resulting formula can be implemented as a computer algorithm. The impatient reader may skip over the next calculations to Theorem \ref{thm:FinalSeries}.

For $n=3$ Theorem \ref{thm:MainThm} gives the expression
\begin{eqnarray}\label{eqn:Formula3}
 3 \sum_{ m_{1}, m_{2} \geq 0  } \frac{(-1)^{m_1+m_2-1}}{m_1+m_2+4}\,m_1 \sum_{l_1+l_2\le m_1+m_2} \textup{fact}(l_1,l_2) \Tr \left( U(m_1,m_2,l_1,l_2)\right).
\end{eqnarray}
Recall from Lemma \ref{lem:U} that the matrix $U(m_1,m_2,l_1,l_2)$ is a sum of words of the form
$$
	s \cdot V^1_1 \ldots V^1_{m_1} W_1 V^2_1 \ldots V^2_{m_2} W_2.
$$
for $s \in \{\pm 1 \}$. We exploit the symmetry between $(m_1,m_2)$ and $(m_2,m_1)$, and the invariance of the trace under cyclic permutations to simplify (\ref{eqn:Formula3}). 
We will only need to consider traces of matrices of the form $U_1\omega_1 U_2 \omega_2$, and in this way we can disregard the sign $s$ of the permutation. Hence we define:
\begin{df}\label{df:Utilde}
The matrix $\widetilde{U}(m_1,m_2,l_1,l_2)$ is the sum of all the matrices of the form
\begin{equation}\label{eqn:Utilde}
	U_1 V^1_1\ldots V^{1}_{m_1} U_2 V^{2}_{1} \ldots V^{2}_{m_2}
\end{equation}
where $V^i_j \in \{ U_k \colon 0\le k \le 2\}$ for all $i,j$ and $l_k = \left|\{ (i,j) \,|\,\ V^i_j = U_k\}\right|$ for $k=1,2$.
\end{df}
\noindent{}To ease notation let us write $c = \frac{(-1)^{m_1+m_2-1}}{m_1+m_2+4}$ whenever $m_1$ and $m_2$ are clear from the context.
\begin{lem}
$$
\sum_{ m_{1}, m_{2} \geq 0  } c\,m_1 \sum_{|\mathbf{l}|\le|\mathbf{m}|} \textup{fact}(\mathbf{l}) \Tr \left( U(\mathbf{m},\mathbf{l})\right) = 
\sum_{ m_{1}, m_{2} \geq 0  } c\,(m_2-m_1) \sum_{|\mathbf{l}|\le|\mathbf{m}|}  \textup{fact}(\mathbf{l}) \Tr ( \widetilde{U}(\mathbf{m},\mathbf{l})).
$$
\end{lem}
\begin{proof}
Fix $m_1, m_2 \ge 0$. Let $\omega_1$ and $\omega_2$ be arbitrary words on letters $U_i$, $0\le i \le 2$ of lenght $m_1$ respectively $m_2$. On the LHS we have the four words 
$$
	\omega_1 U_1 \omega_2 U_2, \quad \omega_1 U_2 \omega_2 U_1, \quad \omega_2 U_1 \omega_1 U_2 \quad \text{and} \quad \omega_2 U_2 \omega_1 U_1.
$$
The trace of the first and last, and of the second and third are the same, say $t$ and $t'$. Hence we have the four summands $c \, m_1\, t$, $-c \, m_1\, t'$, $c \, m_2\, t'$ and $-c \, m_2\, t$, times the factorial coefficient ($\mathbf{l}$ is the same for all words). Suppose that $m_1\neq m_2$. Then on the RHS we have two words, 
$$
	U_1 \omega_1 U_2 \omega_2 \quad \text{and} \quad U_1 \omega_2 U_2 \omega_1.
$$
By a cyclic permutation, the traces are $t'$ and $t$ respectively. The corresponding summands are then $c \, (m_2-m_1)\, t'$ and $c \, (m_1-m_2)\, t$, multiplied by the same factorial coefficient. Finally, the case $m_1=m_2$ gives 0 on both sides. 
\end{proof}


Now consider a matrix $U_1 U_{i_1} \ldots U_{i_k}$ with $i_j \in \{0,1,2\}$ and $k \ge 1$. Write $n_i$ for the total number of letters equal to $U_i$, $i=1,2$ among the $U_{i_j}$.
How many times does this matrix appear in a expression of the form (\ref{eqn:Utilde})?
For each $i_j=2$ it appears in $\widetilde{U}(m_1,m_2,n_1,n_2-1)$ for $m_1 = j-1$ and $m_2 = k-j$, and coefficient 
$$
	\frac{(-1)^{m_1+m_2-1}}{m_1+m_2+4}\,(m_2-m_1)\, \textup{fact}(l_1,l_2) =
	\frac{(-1)^{k-2}}{k+3} \,(k-2j+1)\, \textup{fact}(n_1,n_2-1).
$$
All in all we have proven the following. Write $\chi_a$ for the characteristic function on an integer $a$ (so $\chi_a(a)= 1$ and $\chi_a(i) = 0$ if $i\neq a$).

\begin{thm}[Infinite series for $n=3$]\label{thm:FinalSeries}
Let $X_0, X_1, X_2, X_3 \in GL_N(\mathbb{C})$ and $A=U_0+s_1U_1+s_2U_2$ defined as in Theorem \ref{thm:MainThm} for $n=3$. Suppose that $\norm{A}<1$. 
Then Hamida's function $\varphi(X_0, X_1, X_2, X_3)$ equals the limit of the convergent series
  \begin{eqnarray}\label{eq:FinalSeries} 
	3\, \sum_{k = 1}^\infty \frac{(-1)^{k}}{k+3} \sum_{i_1, \ldots, i_{k} = 0}^{2}
          d \: \textup{fact}(n_1,n_2-1)\: \textup{Tr} \left(U_1U_{i_1} \ldots U_{i_{k}}\right),
  \end{eqnarray}
where $n_1$, $n_2$ and $d$ depend on each tuple $(i_1,\ldots,i_{k})$ as $n_1 =  \sum_j \chi_1(i_j)$, $n_2 = \sum_j \chi_2(i_j)$ and $d  = \sum_{j}\chi_2(i_j)\: (k-2j+1)=n_2(k+1)-2\sum_{j}\chi_2(i_j)\,j$, and $\textup{fact}(l_1,l_2)$ is defined  as in Lemma \ref{lem:fact} for $l_2\ge 0$ and as 0 if $l_2=-1$. 
\end{thm}

We can now explain how to calculate $\varphi(X_0,X_1,X_2,X_3)$ for arbitrary $X_i \in GL_N(\mathbb C)$. Firstly, use Equation (\ref{unitary}) in Remark \ref{rmk:HomoUni} to assume, without loss of generality, that $X_0=I$. Secondly, by Proposition \ref{prop:BoundA}, for any $\mu >0$,
\begin{eqnarray*}
	\varphi(X_0,X_1,X_2,X_3)
		&=& -\; \varphi(\mu X_1,\mu X_2,\mu X_3,I) + \varphi(\mu X_0,\mu X_2,\mu X_3,I)\\ & &-\;\varphi(\mu X_0,\mu X_1,\mu X_3,I)+ \varphi(\mu X_0,\mu X_1,\mu X_2,I).
\end{eqnarray*}
For $\mu$ small enough (see Proposition \ref{prop:BoundA} or Remark \ref{remark:OptimalMu}), each tuple on the right hand side will satisfy the conditions of Theorem \ref{thm:FinalSeries}. In addition, since $X_0=I$, the last three terms satisfy the hypothesis of Lemma \ref{prop:vanishingterm} below, and hence the Hamida function on these tuples vanishes. Therefore
\begin{eqnarray*}
	\varphi(X_0,X_1,X_2,X_3)
		&=& -\; \varphi(\mu X_1,\mu X_2,\mu X_3,I).
\end{eqnarray*}
Finally, we can use the convergent series (\ref{eq:FinalSeries}) to obtain an approximation of the value in the right-hand side.

\begin{lem}\label{prop:vanishingterm}
If $X_0,X_1,X_2,X_3 \in GL_N(\mathbb C)$ satisfy the hypothesis of Theorem \ref{thm:FinalSeries}, and $X_0^\ast{}X_0$ is central (ie.~a scalar multiple of the identity) then $\varphi(X_0,X_1,X_2,X_3)$ vanishes.
\end{lem}
\begin{proof}
This follows easily from Proposition \ref{vanish} (Appendix \ref{comblem}) and Theorem \ref{thm:FinalSeries}.
\end{proof}
All in all, we have the following.
\begin{thm}\label{thm:homologicaltrick3}
Let $X_0,X_1,X_2,X_3 \in GL_N(\mathbb C)$ arbitrary. Define $X_j'=X_jX_0^{-1}$ for each $0\le j \le 3$. Let $\lambda_{\textup{max}} = \max_{0\le j \le n} \lambda_{\textup{max}}\left((X'_j)^\ast X'_j\right)$ and $0 < \mu < \sqrt{\frac{2}{\lambda_\textup{max}}}$. Then
\begin{eqnarray}\label{eqn:simplify}
	\varphi(X_0,X_1,X_2,X_3)
		&=& -\; \varphi(\mu X'_1,\mu X'_2,\mu X'_3,I),
\end{eqnarray}
and the tuple on the right-hand side satisfies the hypothesis of Theorem \ref{thm:FinalSeries}. 
\end{thm}
\begin{cor}\label{cor:vanishing}
Let $X_0,X_1,X_2,X_3 \in GL_N(\mathbb C)$ arbitrary. If $(X_1X_0^{-1})^\ast X_1X_0^{-1}$ is central, then $\varphi(X_0,X_1,X_2,X_3)$ vanishes.
\end{cor}
\begin{proof}
Use (\ref{eqn:simplify}) and Lemma \ref{prop:vanishingterm}.
\end{proof}

\appendix
\section{Computational aspects}\label{section:ComputerAlgorithm}
One advantage of the formula in Theorem \ref{thm:FinalSeries} for Hamida's integral evaluated at $H_3\left(GL(\mathbb C)\right)$ is that it admits a straightforward implementation as a computer algorithm. In this appendix we discuss details relevant to  an implementation of the formula in Theorem \ref{thm:FinalSeries} or the computation of $V_1(F)$ for a cyclotomic field $F$.

The algorithm to compute Hamida's function takes as input matrices $X_0$, $X_1$, $X_2$, $X_3 \in GL_N(\mathbb{C})$, performs the `homological trick' of \S{}\ref{section:HomologicalTrick} and outputs the partial sum of (\ref{eq:FinalSeries}) up to a given $k$. These partial sums converge to Hamida's function $\varphi(X_0,X_1,X_2,X_3)$, and we have an upper bound for the error after $k$ iterations (Appendix \ref{error}). Therefore, we can in principle compute the universal Borel class $b_1$ evaluated at $H_3(GL_N(\mathbb{C}))$ (via Theorem \ref{thm:Hamida}) to any prescribed degree of accuracy. This, together with the input chain $q Z_1 - Z_2(q)$ (see \S\ref{section:basis}) allow the computation of $V_1(F)$ for any cyclotomic field to any required precision. We have included a complete description of the algorithm in Appendix \ref{section:pseudocode}.

However, the computational complexity of this algorithm is exponential on $k$ (Appendix \ref{section:complexity}): the computing time of the $(k+1)$th iteration is roughly 3 times that of the $k$th iteration. The error after $k$ iterations is bounded by a constant times $\|A\|^k \, k$ (see Appendix \ref{error}) so it will nevertheless converge fast to zero, unless $\|A\|$ is close to one. Unfortunately, this seems to be the case in practice: for instance, more than 86\% of the terms in $Z_1\vert_{t=e^{2\pi i/3}}$ have an associated matrix $A$ of norm greater than 0.9, even for the optimal choice of parameter $\mu$ (as in Lemma \ref{optimal}).


A naive implementation with symbolic mathematical software (we used Maple \cite{Maple13}) is not very efficient (it takes on average 12$s$ per term with $k=3$). A numerical implementation (we used Matlab \cite{MatlabR2011b}) is much faster (it takes on average 1.96$s$ per term with $k=8$) but it is still not fast enough to guarantee a small error within reasonable time, according to our error formula. 

It is clear that either a more efficient implementation, or a tighter bound on the error, or an alternative `homological trick' is needed.
In any case, the purpose of the present article is to describe the theory behind a new approach to compute Borel's regulator, and it is our hope that more accurate implementations will be built on these foundations. 

In Appendix \ref{section:Numerical} we will discuss the results of a different (C++) implementation of the same algorithm made by the $1^{\rm st}$ author in the course of his doctoral thesis \cite{ZackyPhD}.

\subsection{Complete algorithm}\label{section:pseudocode}
\fbox{Computation of $\varphi(X_0,X_1,X_2X_3)$ for arbitrary $X_0,X_1,X_2,X_3 \in GL_N(\mathbb C)$}
\begin{enumerate}
\item Define
\[
\begin{array}{rclcrcll}
	\lambda_\text{max} &=& \max_{0 \le i \le 3} \lambda_\text{max}(X_i^*X_i), & \ & Y_i &=& \mu X_i X_0^{-1} & (i=1,2,3),\\
	\lambda_\text{min} &=& \max_{0 \le i \le 3} \lambda_\text{min}(X_i^*X_i), & & U_0 &=& Y_1^\ast Y_1 - I,\\
\mu &=& \sqrt{2/(\lambda_\text{max}+\lambda_\text{min})}, & & U_i &=& Y_i^\ast Y_i - Y_1^\ast Y_1 & (i=2,3).

\end{array}
\]
\item For each $k_\text{max} \ge 1$, consider the finite sum
\begin{eqnarray}\label{eqn:FiniteSum}
	3\, \sum_{k = 1}^{k_\text{max}} \frac{(-1)^{k}}{k+3} \sum_{i_1, \ldots, i_{k} = 0}^{2}
          d(i_1,\ldots,i_{k}) \: \textup{fact}(n_1,n_2-1)\: \textup{Tr} \left(U_1U_{i_1} \ldots U_{i_{k}}\right),
  \end{eqnarray}
with $n_1$, $n_2$, $d$ and $\textup{`fact'}$ defined as in Theorem \ref{thm:FinalSeries}.
Then (\ref{eqn:FiniteSum}) approximates $\varphi( Y_1,  Y_2, Y_3, I)=-\varphi(X_0,X_1,X_2,X_3)$, with an error bounded above by
\[
	 \frac{N \norm{U_1}\norm{U_2}}{(1-\rho)^2} \rho^{k_\text{max}} \left(1+k_{\text{max}}(1-\rho)\right),
\]
where $\rho=\frac{\lambda_\text{max}-\lambda_\text{min}}{\lambda_\text{max}+\lambda_\text{min}}$ (Corollary \ref{cor:Error} and Lemma \ref{lem:Bounds}(ii)).
\end{enumerate}

\bigskip

\noindent\fbox{Computation of $b_1$ evaluated at a term in the inhomogeneous bar resolution}\\[0.5em]
\noindent{}Let  $[g_1|g_2|g_3] \in B_3$ (notation of \S\ref{3.0}).
By writing $[g_1|g_2|g_3]=(I,g_1,g_1g_2,g_1g_2g_3)$ and using Theorem \ref{thm:Hamida} we have that the universal Borel class evaluated at $[g_1|g_2|g_3]$ is
\[
	b_1([g_1|g_2|g_3]) = \frac{1}{16\pi i} \, \varphi\left(I,g_1^\ast,(g_1g_2)^\ast,(g_1g_2g_3)^\ast\right).
\]
We can then compute the right-hand side by the algorithm above. More generally, any element $z \in B_3$ can be written as a finite sum of elements as above.

\bigskip

\noindent\fbox{Computation of $V_1$ for a cyclotomic field}\\[0.5em]
Let  $F=\mathbb Q\left(\xi\right)$ be a cyclotomic field and suppose that $\xi=e^{2\pi i/q}$, for some integer $q \ge 3$.
Let $Z_1$ be the universal chain (recall we have found an explicit representative\footnote{see ancillary file anc/universalchain.mat}) and define $Z_2(q)$ as in Definition \ref{Z2}. Let $\mathcal Z=qZ_1-Z_2(q) \in B_3(\mathbb Z[t,t^{-1}])$. Let and $\pm v_1,\ldots, \pm v_l$ the units in $\mathbb Z/q\mathbb Z$. For each $1 \le i,j \le l$ define $\mathcal Z_{ij}$ by the substitution $t=\xi^{v_i v_j}$ in $\mathcal Z$,
\[
	\mathcal Z_{ij}=\mathcal Z_{|t=\xi^{v_i v_j}} \in B_3(F).
\]
Then $z_{ij}=b_1(\mathcal Z_{ij})$ can be evaluated by our algorithm above. Finally, $V_1(F)$ equals the absolute value of the determinant of the matrix whose $(i,j)$-entry is $z_{ij}$ (Theorem \ref{mothergoose}).

\subsection{Complexity}\label{section:complexity}
The computational complexity of our proposed algorithm (Appendix \ref{section:pseudocode}) is exponential on $k$: for each term in a homological cycle (for the chain $qZ_1 - Z_2(q)$ we have $6844 + 3q$ terms) we compute an instance of the infinite series in Theorem \ref{thm:FinalSeries} (cf.~Theorem \ref{thm:homologicaltrick3}); for each series, we add successive terms until the error is smaller than the target error; finally, the $k^{\rm th}$ summand of the series in Theorem \ref{thm:FinalSeries} involves conducting $3^k$ different multiplications of $k+1$ matrices of size $N$ (for the chain $qZ_1-Z_2(q)$ we have $N=5$).
\begin{remark} For a general $n > 3$ odd we expect an infinite series analogous to the one in Theorem \ref{thm:FinalSeries}. Thus for each term in a homological cycle we would have $n$ instances of the infinite series, and in turn the $k^{\rm th}$ summand in a series would involve $n^k$ multiplications of $k+1$ matrices of size typically $2n+1$.
\end{remark}


\subsection{Error bound}\label{error}
This section is devoted to bounding the error in computing the infinite series (\ref{eq:FinalSeries}) after $k$ iterations. The impatient reader may refer to Proposition \ref{prop:Error} and Corollary \ref{cor:Error}.

\medskip

Throughout this section, let $\norm{\cdot}$ be the spectral norm.
\begin{lem}\label{lem:Lemma1}
Let $X$ be a hermitian matrix of size $N$. Then
$$
	|\textup{Tr}(X)| \le  N \norm{X}\,.
$$
\end{lem}
\begin{proof}
If $\lambda_1,\ldots,\lambda_N$ are the eigenvalues of $X$ then $\norm{X}=\max_i |\lambda_i|$ and hence
\[
    |\textup{Tr}(X)| = \Big|\sum_{i} \lambda_i\Big| \le  N \norm{X}\,. \qedhere
\]
\end{proof}

\begin{lem}\label{lem:TraceVol}
Let $X$ be a matrix of size $N$ of continuous functions over a compact subset $\Delta \subset \mathbb{R}^n$. Suppose that $X(\mathbf{s})$ is hermitian for each $\mathbf{s} \in \Delta$ and define $\norm{X} = \max_{\mathbf{s}\in \Delta} \norm{X(\mathbf{s})}$. Then
$$
\left| \Tr \int_\Delta X(\mathbf{s}) \,d\mathbf{s}  \right| \le \operatorname{vol}(\Delta) \, N \, \norm{X}\,.
$$
\end{lem}
\begin{proof}
By elementary properties of integration and Lemma \ref{lem:Lemma1},
\[
\left| \Tr \int_\Delta X(\mathbf{s}) \,d\mathbf{s}  \right| = \left| \int_\Delta \Tr \left(X (\mathbf{s})\right) d\mathbf{s} \right| \le \int_\Delta \left| \Tr \left( X(\mathbf{s})\right) \right| d\mathbf{s} \le \operatorname{vol}(\Delta) \, N \, \norm{X}\,. \qedhere
\]
\end{proof}
\begin{prop}\label{prop:Error}
Let $a_k$ be the $k^{\rm th}$ summand of the series in Theorem \ref{thm:FinalSeries}, and keep the same hypothesis and notation. Then
\[
	|a_k| \le N \, \norm{U_1} \, \norm{U_2} \, k \,\norm{A}^{k-1}\,.
\]
\end{prop}

\begin{cor} \label{cor:Error}
Let $C=N \, \norm{U_1} \, \norm{U_2}$, $\rho=\norm{A}<1$ and $b_k = k \,\rho^{k-1}$ for each $k \ge 1$. Then 
\[
	\left| \sum_{k=l+1}^\infty a_k \right| \le C \sum_{k=l+1}^\infty b_k =  
	C \, \rho^l \left( \frac{1+l(1-\rho)}{(1-\rho)^2}\right).
\]
\end{cor}
\noindent{}Note that the expression on the left-hand side is the absolute error of approximating $\sum_{k=1}^\infty a_k$ by $\sum_{k=1}^l a_k$, the output of the algorithm after $l$ iterations. The upper bound on the right-hand side converges to zero as $l \to \infty$.
\begin{remark}
Recall that Lemma \ref{lem:Bounds}(ii) allows us to have an explicit value for $\rho=\norm{A}$.
\end{remark}
\begin{proof}[Proof of Proposition]
In order to attain a neat upper bound depending on $\norm{A}$, we retrace our steps in \S\ref{section:formula3} and \S\ref{section:PowerSeriesFormula} to write
\[
	a_k =  \sum_{m_1 + m_2 = k-1 } \frac{(-1)^{k}}{k+3} \, m_1 \Tr \int_{\Delta^2} A^{m_1}\,dA\,A^{m_2}\,dA\,,
\]
since $a_k$ involves all the words of total lenght $k+1$. Recall that $dA = U_1 ds_1 + U_2 ds_2$. Then
\[
	\Tr \int_{\Delta^2} A^{m_1}\,dA\,A^{m_2}\,dA =   \Tr \int_{\Delta^2} A^{m_1}U_1A^{m_2}U_2 ds_1 ds_2 - \Tr \int_{\Delta^2} A^{m_1}U_2A^{m_2}U_1ds_1 ds_2.
\]
By the triangle inequality and the Lemma \ref{lem:TraceVol},
\begin{eqnarray*}
\left| \Tr \int_{\Delta^2} A^{m_1}\,dA\,A^{m_2}\,dA\right| &\le& \operatorname{vol}(\Delta^2)\, N \, \left( \norm{A^{m_1}U_1A^{m_2}U_2} + \norm{A^{m_1}U_2A^{m_2}U_1} \right) \\
	&\le& N \, \norm{U_1} \, \norm{U_2} \, \norm{A}^{k-1}.
\end{eqnarray*}
There are $k$ pairs of nonnegative integers $(m_1,m_2)$ such that $m_1+m_2 = k-1$. Hence 
$$
	|a_k| \le k \, \frac{1}{k+3}\, k \, N \, \norm{U_1} \, \norm{U_2} \, \norm{A}^{k-1} \,.
$$
Observing that $\frac{k}{k+3} \le 1$ finishes the proof.
\end{proof}

\subsection{Vanishing of terms where $X_0^\ast{}X_0$ is central}\label{comblem}
Recall that Theorem \ref{thm:homologicaltrick3} and Corollary \ref{cor:vanishing} depend on Lemma \ref{prop:vanishingterm}, which is in turn a consequence of the following result.

\begin{prop} \label{vanish} 
Keep the notation and hypothesis of Theorem \ref{thm:FinalSeries}. Let $k \ge 1$, $n_1, n_2 \ge 0$ be integers.  If $X_0^\ast{}X_0$ is central then
\begin{eqnarray} \label{SSum} 
	\sum_{(i_1, \ldots, i_{k})\in I}
          d(i_1, \ldots, i_{k}) \,  \textup{Tr} \left(U_1U_{i_1} \ldots U_{i_{k}}\right)=0,
 \end{eqnarray}
\noindent where $I$ denotes the set of all sequences of length $k$ which have $n_1$ terms equal to $1$, $n_2$ terms equal to $2$, and the remaining terms equal to $0$, and $d \colon I \to \mathbb Z$ is defined as $ d(i_1, \ldots, i_{k})=n_2(k+1)-2\sum_{i_j=2}j$.
\end{prop}
\noindent As a corollary, the Hamida function on the tuple will vanish: simply collate the terms in the sum (\ref{eq:FinalSeries}) with fixed $k$, $n_1$ and $n_2$. This proves Lemma \ref{prop:vanishingterm}.

\medskip

The rest of this section consists on a combinatorial proof of Proposition \ref{vanish}.

\medskip

For every sequence $(i_1,\cdots i_k) \in I$ we define $T(i_1,\cdots i_k)=(1,i_{j_1},i_{j_2},\cdot,i_{j_{n_1+n_2}})$, where the $i_{j_t}$ form the subsequence of all non-zero terms, in the same order as they appear in $(i_1,\cdots i_k)$.  That is, the operation $T$ adds a $1$ to the start of the sequence and removes all the $0$'s.

We partition $I$ into equivalence classes by setting $i \sim i'$ precisely when $T(i)$ is a cyclic permutation of $T(i')$.  Our strategy in proving Proposition \ref{vanish} will be to show that for any equivalence class $J \subset I$, we have: 

\begin{eqnarray} 
	\sum_{(i_1, \ldots, i_{k})\in J}
          d(i_1, \ldots, i_{k}) \,  \textup{Tr} \left(U_1U_{i_1} \ldots U_{i_{k}}\right)=0.
 \end{eqnarray}

\begin{lem}
The traces $\textup{Tr} \left(U_1U_{i_1} \ldots U_{i_{k}}\right)$ in the sum above are all equal.
\end{lem}

\begin{proof}
The trace of a product of matrices is invariant under cyclic permutations of the product.  Further, as $U_0$ is central, the trace is not effected by the exact position of the $U_0$ terms in the product.

Given two traces $\textup{Tr} (U_1U_{i_1} \ldots U_{i_{k}})$ and $\textup{Tr} (U_1U_{i_1'} \ldots U_{i_{k}'})$ in the above sum, we may cycle the product of matrices in the first one to get the the second one, if we ignore the $U_0$ terms. Since both products have the same number $k-n_1-n_2$ of $U_0$ terms, and they are central, both traces coincide.
\end{proof}
\noindent{}Therefore it then remains to prove the following purely combinatorial result:
\begin{eqnarray} \label{eqn:combres}
	\sum_{(i_1, \ldots, i_{k})\in J}
          d(i_1, \ldots, i_{k}) =0.
 \end{eqnarray}
\noindent We first examine $d$ more closely:
\begin{equation}\label{eqn:d}
d(i_1, \ldots, i_{k})\,\,=\,\,n_2(k+1)-2\sum_{i_j=2}j \,\,=\,\,\sum_{i_j=2} \left(k+1-2j\right)\,\,=\,\,\sum_{i_j=2}\sum_{x=1}^k \sigma_j(x)
\end{equation}
where 
\[
\sigma_j(x)=\left \{ \begin{array}{rc}-1& \text{if } x<j, \\ 0& \text{if } x=j ,\\ 1& \text{if } x>j.  \end{array}\right.
\]
(For the last equality in (\ref{eqn:d}), note that $\sum_{x=1}^k \sigma_j(x)=(-1)(j-1)+1(k-j)=k-2j+1$.)\\
Fix a sequence $i \in J$ and let $T(i)=(a_1,a_2,\cdots, a_{m})$.  Then $m=n_1+n_2+1$ and each $a_t$ lies in $\{1,2\}$.  Let $s$ denote the smallest non-negative integer such that $a_t=a_{t+s}$ for all $t$, where the indices are taken modulo $m$.  That is, the order of the group of cyclic symmetries of $T(i)$ is $m/s$. 

Now let $C$ denote the set of $m$ dimensional vectors with non-negative integer entries which sum to $k-n_1-n_2$. (Note that $C$ is a finite set.)

To each integer $l$, $1\le l\le m$, such that $a_l=1$ and to each vector $\vec{c}=(c_1,\cdots,c_m) \in C$, we associate a sequence $\theta(l,\vec{c}) \in J$ constructed as follows:

\begin{itemize}
\item[1)] Starting with the sequence $(a_1,\cdots,a_m)$, insert $c_t$ zeroes after $a_t$, for each $t$.

\item[2)] Cycle the resulting sequence of length $k+1$ so that $a_l=1$ is the first term.

\item[3)] Delete the first term (which is $a_l=1$).
\end{itemize}


\begin{lem}
Every element of $J$ can be constructed in precisely $m/s$ ways as $\theta(l,\vec{c})$ for an integer $l$ with $a_l=1$ and a vector $\vec{c}\in C$.  
\end{lem}

\begin{proof}
Given a sequence $j \in J$, which we wish to construct by the operation above, we have a choice of $m/s$ integers for $l$.  And for each such choice there is a unique choice of vector $\vec{c} \in C$ which results in $j$.
\end{proof}

Therefore we can write
\begin{eqnarray} \label{quadsum}
	\sum_{(i_1, \ldots, i_{k})\in J}
          d(i_1, \ldots, i_{k}) =  \frac sm \sum_{\vec{c}\in C}\sum_{a_l=1} d\left(\theta(l,\vec{c})\right) \stackrel{(\ref{eqn:d})}{=} \frac sm \sum_{\vec{c}\in C}\sum_{a_l=1}\sum_{a_j=2} \sum_{x=1}^k \sigma_{b_j}(x),
 \end{eqnarray}
where in each case $b_j$ denotes the position of the term $a_j$ in the sequence $\theta(l,\vec{c})$.

Given integers $r,t,u$ modulo $m$, we set: 
\[
\epsilon_{rtu}=\left \{ \begin{array}{rc}-1 & {\rm if\, starting\, at\,} r\, {\rm and\, repeatedly\, adding\,} 1{\rm, one\, arrives\, at\,} u\, {\rm before\,} t, \\ \,\,0 & {\rm if}\, r,t,u \, {\rm are \, not\,pairwise\,distinct,}\hfill   \\ \,\,1 & {\rm if \,starting\, at\,} r\, {\rm and\, repeatedly\, adding\,} 1{\rm, one\, arrives\, at\,} t\, {\rm before\,} u.   \end{array}\right.
\]
To the sequence $(a_1,\cdots,a_m)$ we associate two $m$ dimensional vectors $\vec{e}=(e_1,\cdots,e_m)$ and $\vec{f}=(f_1,\cdots,f_m)$.  They are defined as
$$e_t= \sum_{a_r=1}\sum_{a_u=2} \epsilon_{rtu}, \qquad f_t= \left\{\begin{array}{cc} e_t+ n_2 & a_t=1,\\e_t-(n_1+1) & a_t=2. \end{array}\right.
$$
One may visualize these definitions by arranging the $a_t$ round an oriented circle and noting that $e_t$ counts the arcs from a 1 to a 2 passing through $a_t$ (with sign given by the direction), and $f_t$ counts the arcs from a 1 to a 2 containing the arc between $a_t$ and $a_{t+1}$ (again with sign given by the direction). As always, when dealing with the $a_t$, we regard the indices modulo $m$, and recall that there is $n_1+1$ 1's and $n_2$ 2's in the sequence $(a_1,\ldots,a_m)$.

Let $e=\sum_{t=1}^m e_t$ and note that $$\sum_{t=1}^m f_t=e+(n_1+1)n_2-n_2(n_1+1)=e.$$ 

Now by symmetry we have $\sum_{\vec{c}\in C} \vec{c}= r\,(1,\cdots,1)$, for some integer $r$.  We may separate the terms in the final sum in (\ref{quadsum}) into those where the $x^{\rm th}$ term in $\theta(l,\vec{c})$ is zero, and those where it is not:

\begin{eqnarray*} 
	\sum_{(i_1, \ldots, i_{k})\in J}
          d (i_1, \ldots, i_{k}) =   \frac sm \sum_{\vec{c}\in C}\sum_{a_l=1}\sum_{a_j=2} \sum_{x=1}^k \sigma_{b_j}(x)
=\frac sm \sum_{\vec{c}\in C}\sum_{t=1}^m (c_t f_t+e_t)\\ =\frac sm (re+|C|e).
 \end{eqnarray*}

\noindent Thus we may complete the proof of (\ref{eqn:combres}) by showing that $e=0$.  That is:

\begin{lem}
Suppose a finite number of red and blue dots are arranged round a circle.  Let \emph{segment} denote a portion of the circle between adjacent dots.  Consider the set of all directed arcs going from a blue dot to a red dot.  Then the total number $e$ of segments traversed by the arcs (counting with negative sign if counterclockwise and positive sign if clockwise) is 0. 
\end{lem}

\noindent (Note we have replaced 1's and 2's with blue dots and red dots respectively).

\begin{proof}
First consider a diameter of the circle and suppose that all the blue dots lie on one side of the diameter and all the red dots lie on the other side.  Then by symmetry we would have $e=0$.

It remains to show that $e$ is constant under interchanging an adjacent red and blue dot.  Let $e_1$ be the value of $e$ when the blue dot is clockwise of the red dot (state 1) and let $e_2$ be the value of $e$ when they are interchanged, so the red dot is clockwise of the blue dot (state 2).  As before, let $m$ denote the total number of dots.

Then every dot other than the pair we are interchanging, has two arcs joining it to one of the pair.  Each of these arcs becomes one segment larger going from state 1 to state 2 (bearing in mind that we are counting segments with sign).  Thus $e_2$ has an extra $2(m-2)$ more than $e_1$ coming from these arcs.

The contributions from the arcs not involving either of the interchanged pair are not effected.  Thus the only other consideration is the two arcs joining the pair.  In state 1 they contribute $-1$ and $m-1$ to the sum.  In state 2 they contribute $1$ and $1-m$ to the sum.

Thus we have $e_2=e_1+2(m-2) -(m-2)+(2-m)=e_1$, as required. 
\end{proof}

%
%

\section{A numerical computation of  $V_1(F)$  \label{section:Numerical}}

In the course of his doctoral studies \cite{ZackyPhD}, the $1^{st}$ author carried out his own C++ implementation of an algorithm based on Theorem \ref{thm:FinalSeries} to compute $V_1(F)$ in the case where $F$ is the cyclotomic field  $\mathbb{Q}(\omega)$ with $\omega=e^{\frac{2 \pi i}{3}}$.  
We will now briefly discuss the numerical results of \cite{ZackyPhD} based on our theory.

There is just one pair of conjugate embeddings $\psi_1,\overline{\psi}_1\colon F \hookrightarrow \mathbb{C}$, and one
pair of conjugate primitive $3^{\rm rd}$ roots of unity, $\omega,\,\omega^2\in F$.  Thus $V_1(F)$ is the absolute value of the determinant of the one by one  matrix whose entry (see \S\ref{section:basis}) is:  $$L_{11}=b_1 \left([(3Z_1-Z_2(3))|_{t=\omega}]\right)/i \,.$$   
This has 6844 terms coming from $Z_1$ and 9 terms coming from $Z_2(3)$.  In \S5.1 of \cite{ZackyPhD}, it is explained that the program first removed terms containing repeated matrices, as they do not contribute to $L_{11}$, and combined terms with the same matrices in permuted order (see Remark \ref{glutus:maximus}).  
This left 3450 distinct terms.   The series formula of Theorem \ref{thm:FinalSeries} was applied to these terms.  After running for 12 hours on a 2.4GHz computer the estimated value of $V_1(F)$, up to two decimal points, was
$$
 \ V_1(F)=0.02\,.
$$
Assuming that the error is in fact no larger than 0.01, we would know that $V_1(F)\neq 0$ and, from Theorem \ref{almost there}, that $R_1(F)=V_1(F)/{\rm ind}(F)$ with $V_1(F)$ lying in the range $[0.01,0.03]$, and ${\rm ind}(F) \in \mathbb{Z}$.

For this number field the Borel regulator $R_1(F)$ is known to be $-2^k\, 3\, \zeta_F^\ast(-1)$ for some $k \in \mathbb{Z}$, and $\zeta_F^\ast(-1) = -0.02692\ldots$ (see \S5.2 in \cite{ZackyPhD} for details). 

If $\textup{ind}(F)=1$, then we would have $R_1(F)=V_1(F)$, and the calculation would  suggest that $k=-2$, as this would give $R_1(F)=\frac{3}{4} \zeta_F^\ast(-1)=0.02019\ldots$.  However, even assuming that the error is bounded by 0.01, this is still just big enough to allow $k=-3$. Of course reducing the size of the error would eliminate this possibility.  This would require a faster implementation than the one in \cite{ZackyPhD}.

\section*{Acknowledgements}
We would like to warmly thank Herbert Gangl for his helpful advice and detailed comments on a previous version of our paper. We would also like to thank MathOverflow user `Ralph' for answering a question about the Hurewicz homomorphism, Rob Snocken and Christopher Voll for explaining various aspects of Dedekind zeta functions to us, Bernhard Koeck for useful discussions on $K$-theory, Matthias Flach for his comment on the orders of some K-theory groups and the Algebra and Number Theory Group at University College Dublin, in particular Robert Osburn, for their invitation to discuss this work.

Finally we would like to thank EPSRC for funding this research, Sajeda Mannan for facilitating a research visit, and the universities of Sheffield and Southampton for supporting this work.

\bibliographystyle{amsplain}
\bibliography{BiblioCMSS}

\end{document}